\documentclass[12pt, a4paper]{amsart}

\usepackage[margin=1.3in]{geometry}

\usepackage{graphicx}
\usepackage{amsmath}
\usepackage{amscd}
\usepackage{amsfonts}
\usepackage{amssymb}
\usepackage{amsthm}
\usepackage[dvipsnames]{xcolor}
\usepackage{mathtools}
\usepackage{todonotes}
\usepackage{tikz-cd}
\usepackage{nameref}

\usepackage[hypertexnames=false, colorlinks, citecolor=red, linkcolor=blue, urlcolor=red, bookmarksopen]{hyperref}

\usepackage{enumitem}
\newcounter{dummy}
\makeatletter
\newcommand\myitem[1][]{\item[#1]\refstepcounter{dummy}\def\@currentlabel{#1}}
\makeatother

\usepackage[depth=3]{bookmark}
\let\oldtocsection=\tocsection
\let\oldtocsubsection=\tocsubsection
\let\oldtocsubsubsection=\tocsubsubsection
\renewcommand{\tocsection}[2]{\hspace{0em}\vspace{0.5mm}\oldtocsection{#1}{#2}}
\renewcommand{\tocsubsection}[2]{\hspace{1.8em}\vspace{0.5mm}\oldtocsubsection{#1}{#2}}
\renewcommand{\tocsubsubsection}[2]{\hspace{4.4em}\vspace{0.5mm}\oldtocsubsubsection{#1}{#2}}

\numberwithin{equation}{section}

\newtheorem{theorem}{Theorem}[section]
\newtheorem{lemma}[theorem]{Lemma}
\newtheorem{proposition}[theorem]{Proposition}

\newtheorem{corollary}[theorem]{Corollary}

\theoremstyle{definition}
\newtheorem{example}[theorem]{Example}
\newtheorem{remark}[theorem]{Remark}
\newtheorem{definition}[theorem]{Definition}
\newtheorem{problem}[theorem]{Problem}

\newcommand{\foral}{\text{ for all }}
\newcommand{\AND}{\text{ and }}

\begin{document}
	
	
	\title{Shift-cyclicity in Analytic Function Spaces}
	
	\author{Jeet Sampat}
	\address{Faculty of Mathematics, Technion - Israel Institute of Technology,	Haifa}
	\email{sampatjeet@campus.technion.ac.il}
	
	\subjclass[2020]{47A16, 46E15, 32A10}
	
	
	\begin{abstract}
        In this survey, we consider Banach spaces of analytic functions in one and several complex variables for which: (i) polynomials are dense, (ii) point-evaluations on the domain are bounded linear functionals, and (iii) the shift operator corresponding to each variable is a bounded linear map. We discuss the problem of determining the shift-cyclic functions in such a space, i.e., functions whose polynomial multiples form a dense subspace. This problem is known to be intimately connected to some deep problems in other areas of mathematics, such as the dilation completeness problem and even the Riemann hypothesis.

        \vspace{2mm}
        
        What makes determining shift-cyclic functions so difficult is that often we need to employ techniques that are specific to the space in consideration. We therefore cover several different function spaces that have frequently appeared in the past such as the Hardy spaces, Dirichlet-type spaces, complete Pick spaces and Bergman spaces. We highlight the similarities and the differences between shift-cyclic functions among these spaces and list some important general properties that shift-cyclic functions in any given analytic function space must share. Throughout this discussion, we also motivate and provide a large list of open problems related to shift-cyclicity.
		
		\vspace{4mm}
		
		\noindent \textbf{Keywords.} Shift operator, Invariant subspaces, Cyclic vectors, Hardy spaces, Dirichlet space, Bergman space, Drury--Arveson space, Complete Pick space.
	\end{abstract}
	
	\maketitle

    \newpage
    
    \tableofcontents
	
	\section{Introduction}\label{sec:intro}
	
		Let $\mathcal{X}$ be a separable Banach space and let $\mathcal{L}(\mathcal{X})$ denote the space of all bounded linear self maps of $\mathcal{X}$.
		For a given $T \in \mathcal{L}(\mathcal{X})$, we say that a closed subspace $\mathcal{Y} \subseteq \mathcal{X}$ is a \emph{$T$-invariant subspace} if
        \begin{equation*}
            T(\mathcal{Y}) := \{ Tv : v \in \mathcal{Y} \} \subseteq \mathcal{Y}
        \end{equation*}
		If $\mathcal{X}$ is finite dimensional, then the eigenspaces corresponding to $T$ are clearly all $T$-invariant.
		Now, if $\mathcal{X}$ is a complex vector space, then we know that such eigenspaces always exist and therefore, any linear map on a finite dimensional complex Banach space has a non-trivial invariant subspace.
		Thus, it is natural to ask if one can similarly always find closed non-trivial invariant subspaces for any given $T \in \mathcal{L}(\mathcal{X})$ even when $\mathcal{X}$ is infinite dimensional.
		This is the famous \emph{invariant subspace problem (ISP)}\footnote{ISP was inspired by unpublished works of Beurling and von Neumann circa mid-$20^{\text{th}}$ century. Halmos \cite{Hal63} then popularized the ISP through expository writing.}, which has been resolved in specific cases (see \cite{AH11, AS54, Atz83, Bea85, BR66, Enf75, Hal66, Lom73, Rea85, Sli08}), but remains open -- as of the writing of this article -- for the case when $\mathcal{X}$ is a separable Hilbert space.
		\begin{problem}[ISP]\label{prob:inv.subsp.prob}
			Given any separable Hilbert space $\mathcal{X}$ and some $T \in \mathcal{L}(\mathcal{X})$, does there always exist a closed subspace $\{0\} \subsetneq \mathcal{Y} \subsetneq \mathcal{X}$ that is $T$-invariant?
		\end{problem}
		
		One way to generate $T$-invariant subspaces is by taking $0 \neq v \in \mathcal{X}$ and defining
		\begin{equation*}
			T[v] := \overline{\operatorname{span}}\left\{ v, Tv, T^2v, \dots \right\},
		\end{equation*}
		where $\overline{\operatorname{span}}$ denotes the closed linear hull in $\mathcal{X}$.
		It is clear that
        \begin{equation*}
            T(T[v]) \subseteq T[v]
        \end{equation*}
        and thus, $T[v]$ is $T$-invariant.
		We call $T[v]$ the \emph{$T$-invariant subspace generated by $v$}.
		In order to solve the ISP, one needs to guarantee that there always exists some $0 \neq v \in \mathcal{X}$ for which $T[v] \subsetneq \mathcal{X}$, however, this is not at all clear when $\mathcal{X}$ is infinite dimensional.

		This motivates us to define a \emph{$T$-cyclic vector}, which is any $v \in \mathcal{X}$ such that $T[v] = \mathcal{X}$.
		If there exists a $T$-cyclic vector in $\mathcal{X}$, then we say that \emph{$T$ is cyclic on $\mathcal{X}$}.
		The following result should now be clear from these definitions.
		\begin{proposition}\label{prop:eqv.ISP.formulation.using.cyc.vectors}
			If $\mathcal{X}$ and $T$ are as above, then the following are equivalent:
			
			\begin{enumerate}
				\item There are no closed non-trivial $T$-invariant subspaces in $\mathcal{X}$.
				
				\item Every $0 \neq v \in \mathcal{X}$ is $T$-cyclic.
			\end{enumerate}
		\end{proposition}
		
	    Evidently, determining the cyclicity of a vector is challenging in general. For finite dimensional spaces, however, determining whether an operator is cyclic is a fairly standard exercise in linear algebra based on the following result
		(see \cite[the second corollary to Theorem 3 in Section 7.2]{HK73}).
		
		\begin{theorem}\label{thm:cyclic.op.fin.dim.case}
			Let $T$ be a matrix operator on a finite dimensional Banach space $\mathcal{X}$. Then, $T$ is cyclic on $\mathcal{X}$ if and only if the minimal and the characteristic polynomial of $T$ are the same.
		\end{theorem}
		If $T$ is cyclic on a finite dimensional space $\mathcal{X}$, then there exist computational algorithms to generate the $T$-cyclic vectors as well (see \cite[Sections 5 and 6]{AC97}).
		This is not true in general if $\mathcal{X}$ is infinite dimensional, and we typically need to employ other tools -- specific to the space -- in order to determine cyclic vectors.
		We showcase this with a concrete example in the next subsection.

	\subsection{Smirnov--Beurling theorem}\label{subsec:Smirnov-Beurling.thm}
	
		The simplest infinite dimensional space is
		\begin{equation*}
			\ell^2(\mathbb{Z}_+) := \left\{ (a_k)_{k \in \mathbb{Z}_+} \subset \mathbb{C} : \sum_{k \in \mathbb{Z}_+} |a_k|^2 < \infty \right\},
		\end{equation*}
		where $\mathbb{Z}_+$ denotes the set of all non-negative integers.
		$\ell^2(\mathbb{Z}_+)$ turns into a Hilbert space with the inner product
		\begin{equation}\label{eqn:l^2.ip}
			\left\langle (a_k)_{k \in \mathbb{Z}_+}, (b_k)_{k \in \mathbb{Z}_+} \right\rangle_{\ell^2(\mathbb{Z}_+)} := \sum_{k \in \mathbb{Z}_+} a_k \overline{b_k},
		\end{equation}
		and has a Schauder basis $\mathcal{B} := \left\{ e^{(0)}, e^{(1)}, \dots \right\}$, where $e^{(j)}$'s have $1$ in the $j^{\text{th}}$ position and $0$'s elsewhere. Consider the \emph{$\ell^2$-shift} $S_{\ell^2} \in \mathcal{L}(\ell^2(\mathbb{Z}_+))$ given by
		\begin{equation*}
			S_{\ell^2}((a_0, a_1, \dots)) := (0, a_0, a_1, \dots).
		\end{equation*}
		Clearly, $S_{\ell^2} e^{(j)} = e^{(j+1)}$ for each $j \in \mathbb{Z}_+$ and hence, $S_{\ell^2}$ is a cyclic operator with $e^{(0)}$ as an $S_{\ell^2}$-cyclic vector. It turns out that determining other $S_{\ell^2}$-cyclic vectors requires one to think outside the box.
		
		Let $\mathbb{D}$ denote the \emph{unit disk} in the complex plane $\mathbb{C}$ and let $\operatorname{Hol}(\mathbb{D})$ be the vector space of holomorphic functions on $\mathbb{D}$. Now, we define
		\begin{equation}\label{eqn:def.H^2(D)}
			H^2(\mathbb{D}) := \left\{ f(z) = \sum_{k \in \mathbb{Z}_+} a_k z^k \in \operatorname{Hol}(\mathbb{D}) : \sum_{k \in \mathbb{Z}_+} |a_k|^2 < \infty \right\}.
		\end{equation}
		Clearly, $H^2(\mathbb{D})$ is also a Hilbert space with the inner product
		\begin{equation*}
			\left\langle f = \sum_{k \in \mathbb{Z}_+} a_k z^k, g = \sum_{k \in \mathbb{Z}_+} b_k z^k \right\rangle_{H^2(\mathbb{D})} := \sum_{k \in \mathbb{Z}_+} a_k \overline{b_k},
		\end{equation*}
		and is unitarily equivalent to $\ell^2(\mathbb{Z}_+)$ via the correspondence $z^k \leftrightarrow e^{(k)}$. Under this correspondence, note that $S_{\ell^2}$ becomes the \emph{Hardy shift}
		\begin{equation*}
			S_{H^2} f(z) := z f(z),
		\end{equation*}
		and that $S_{\ell^2}$-cyclic vectors correspond to $S_{H^2}$-cyclic functions. Smirnov \cite{Smi32} and Beurling \cite{Beu49} independently identified $S_{H^2}$-cyclic functions.
		\begin{theorem}\label{thm:Beurling.thm.H^2}
			A function $f \in H^2(\mathbb{D})$ is $S_{H^2}$-cyclic if and only if
			\begin{equation}\label{eqn:outer.func.H^2}
				f(0) \neq 0 \AND \log |f(0)| = \int_0^{2 \pi} \log |f^*(e^{i \theta})| \frac{d \theta}{2 \pi}.
			\end{equation}
		\end{theorem}
		Here, $f^*(e^{i \theta})$ denotes the \emph{radial limit}
        \begin{equation*}
            f^*(e^{i \theta}) := \lim_{r \to 1^-} f(r e^{i \theta}),
        \end{equation*}
        which is known to exist for a.e. $\theta \in [0,2\pi)$ and for each $f \in H^2(\mathbb{D})$ (see Theorem \ref{thm:radial.limits.H^p(D)}). A function satisfying \eqref{eqn:outer.func.H^2} is called an \emph{outer function}. Outer functions play a crucial role in understanding cyclicity beyond $H^2(\mathbb{D})$.	Note that outerness cannot be realized as a property of the power-series coefficients and thus, the Smirnov--Beurling theorem exhibits an interesting overlap between operator theory and function theory. It is also important to note that such a result is rare even for general function spaces over $\mathbb{D}$. Nevertheless, the Smirnov--Beurling theorem spearheaded a deeper analysis of the \emph{shift operator} $S : f \mapsto z f$, its invariant subspaces and its cyclic vectors.

	\subsection{Cyclicity in general}\label{subsec:cyclicity}
	
		In this article, we shall mostly study cyclicity with respect to the shift operator $S : f \mapsto z f$ and its generalization to function spaces in several variables. However, it will be convenient to provide the most general setting in which one can study cyclicity. Let $\mathcal{X}$ be a Banach space and $(\mathcal{G}, \cdot)$ be a semigroup. A family
        \begin{equation*}
            \tau = \left\{ T_g \right\}_{g \in \mathcal{G}} \subset \mathcal{L}(\mathcal{X})
        \end{equation*}
        is called \emph{$\mathcal{G}$-compatible} if
		\begin{equation*}
			T_g \circ T_h = T_{g \cdot h} \foral g,h \in \mathcal{G}.
		\end{equation*}
		We define the \emph{$\tau$-invariant subspace generated by $v \in \mathcal{X}$} to be
		\begin{equation*}
			\tau[v] := \overline{\operatorname{span}}\left\{ T_g v : g \in \mathcal{G} \right\},
		\end{equation*}
		and say that $v$ is \emph{$\tau$-cyclic} if $\tau[v] = \mathcal{X}$.
		
		\begin{example}\label{eg:tau-cyclicity}
			
			Let us mention a few important examples of $\tau$-cyclicity.
			\begin{enumerate}
				\item $S_{H^2}$-cyclicity is equivalent to $\tau$-cyclicity if $(\mathcal{G},\cdot) = (\mathbb{Z}_+,+)$ and
				\begin{equation*}
					\tau = \left\{ S_{H^2}^k : k \in \mathbb{Z}_+ \right\}.
				\end{equation*}
				
				\item Consider
				\begin{equation*}
					H^2_0(\mathbb{D}) := \left\{ f \in H^2(\mathbb{D}) : f(0) = 0 \right\} = \overline{\operatorname{span}}^{H^2(\mathbb{D})}\left\{ z^n : n \in \mathbb{N} \right\}
				\end{equation*}
				and let $(\mathcal{G},\cdot) = (\mathbb{N},\times)$. Now, define the \emph{power dilations} $T_n \in \mathcal{L}(H^2_0(\mathbb{D}))$
				\begin{equation*}
					T_n f(z) = f(z^n) \foral n \in \mathbb{N}
				\end{equation*}
				and let
				\begin{equation*}
					\tau = \left\{ T_n : n \in \mathbb{N} \right\}.
				\end{equation*}
				Clearly, $\tau$ is $\mathbb{N}$-compatible and the identity map $\operatorname{id}_{\mathbb{D}} : z \mapsto z$ is $\tau$-cyclic.
				
				\item For $L^2(0,\infty)$, we define the \emph{dilations} $D_n \in \mathcal{L}(L^2(0,\infty))$
				\begin{equation*}
					D_n \varphi(x) = \varphi(nx) \foral n \in \mathbb{N}
				\end{equation*}
				and let
				\begin{equation*}
					\vartheta = \left\{ D_n : n \in \mathbb{N} \right\}.
				\end{equation*}
				If we take $(\mathcal{G},\cdot) = (\mathbb{N}, \times)$ as before, then $\vartheta$ is $\mathbb{N}$-compatible. It turns out that $\vartheta$-cyclicity is incredibly complex as the following theorem suggests.
			\end{enumerate}
		\end{example}
		
		\begin{theorem}\label{thm:NBBD.criterion}
			Let $\{t\}$ be the fractional part of $t \in \mathbb{R}$ and $\chi_{[0,1]}$ be the indicator function of $[0,1]$. Then, the Riemann hypothesis is true if and only if
			\begin{equation}\label{eqn:NBBD.criterion}
				\chi_{[0,1]} \in \vartheta[\varrho],
			\end{equation}
            where $\varrho \in L^2(0,\infty)$ is the function
            \begin{equation*}
                \varrho(x) := \left\{ \frac{1}{x} \right\}.
            \end{equation*}
		\end{theorem}
		
		\begin{remark}\label{rem:NBBD.criterion}
			This theorem is due to B{\'a}ez-Duarte \cite{Bae03}, who strengthened a similar criterion given by Beurling \cite{Beu55}, who in turn strengthened a result of his student Nyman from his Ph.D. thesis \cite{Nym50}. Thus, \eqref{eqn:NBBD.criterion} is called the \emph{Nyman--Beurling--B{\'a}ez-Duarte criterion (NBBD)} for the Riemann hypothesis.
		\end{remark}
		
		Somewhat surprisingly, all notions of cyclicity mentioned in Example \ref{eg:tau-cyclicity} are related to each other -- albeit loosely. We give a snippet of this relationship in Section \ref{subsubsec:func.th.junc}. For now, we refer the reader to the following excellent references on the \emph{Euler $\zeta$ function}, the Riemann hypothesis and NBBD \cite{Bag06, Edw01, HLS97, Nik19, Noo19}.
		
	\subsection{Shift-cyclicity}\label{subsec:Shift.cyc}
	
		Let us first set up some important notation that will be used throughout our discussion.
	
		\subsubsection{Analytic function spaces}\label{subsubsec:anal.func.sp} Fix $d,l \in \mathbb{N}$ and some open set $\Omega \subset \mathbb{C}^d$. A function $f : \Omega \to \mathbb{C}^l$ is said to be \emph{holomorphic on $\Omega$} if it is continuous on $\Omega$ and holomorphic in each variable separately. We write $\operatorname{Hol}(\Omega)$ for the space of all $\mathbb{C}$-valued holomorphic functions on $\Omega$. The domains of primary interest for us are the \emph{unit polydisk}
		\begin{equation*}
			\mathbb{D}^d := \left\{ w = (w_1,\dots,w_d) \in \mathbb{C}^d : \max_{1 \leq j \leq d} |w_j| < 1 \right\}
		\end{equation*}
		and the \emph{unit ball}
		\begin{equation*}
			\mathbb{B}_d := \left\{ w = (w_1,\dots,w_d) \in \mathbb{C}^d : \sum_{j = 1}^d |w_j|^2 < 1 \right\}.
		\end{equation*}
		
		From now on, $z = (z_1,\dots,z_d)$ shall denote both, the $d$-tuple of complex variables and also the $d$-tuple of functions
        \begin{equation*}
            z_j : w \mapsto w_j.
        \end{equation*}
        Let $\mathbb{Z}_+^d$ be the collection of all $d$-tuples of non-negative integers. For a given $\alpha = (\alpha_1,\dots,\alpha_d) \in \mathbb{Z}_+^d$, we use the \emph{multinomial notation}
		\begin{equation*}
			z^\alpha = z_1^{\alpha_1} z_2^{\alpha_2} \dots z_d^{\alpha_d}
		\end{equation*}
		and denote the space of all \emph{analytic polynomials in $d$ variables} as
		\begin{equation*}
			\mathcal{P}_d = \operatorname{span} \left\{ z^\alpha : \alpha \in \mathbb{Z}_+^d \right\}.
		\end{equation*}
		For $\Omega = \mathbb{D}^d$ or $\mathbb{B}_d$, every $f \in \operatorname{Hol}(\Omega)$ has a \emph{(global) power-series representation}
		\begin{equation*}
			f(z) = \sum_{\alpha \in \mathbb{Z}_+^d} c_\alpha z^\alpha,
		\end{equation*}
		and the series converges absolutely and uniformly on all compact subsets of $\Omega$ (see \cite[Section 1.1.3]{Rud69} and \cite[Section 1.2.6]{Rud80}). Thus, we can endow $\operatorname{Hol}(\Omega)$ with the topology of uniform convergence on compact subsets of $\Omega$. Henceforth, $\mathcal{X} \subset \operatorname{Hol}(\Omega)$ will always represent a linear subspace with a topology such that convergence in $\mathcal{X}$ implies convergence in the topology of uniform convergence on compact subsets. In fact, we can guarantee that this holds by introducing other properties that $\mathcal{X}$ satisfies. These properties are important to even make sense of the notion of shift-cyclicity.
		
		\begin{definition}\label{def:Banach.sp.anal.func}
			A vector space $\mathcal{X} \subset \operatorname{Hol}(\Omega)$ is said to be a \emph{Banach/Hilbert space of analytic functions} if it is a Banach/Hilbert space with norm $\|\cdot\|_\mathcal{X}$ that satisfies:
			
			\begin{enumerate}
				\myitem[\textbf{P1}] $\mathcal{P}_d$ is dense in $\mathcal{X}$.\label{item:P1}
				
				\myitem[\textbf{P2}] The \emph{evaluation functional}
                \begin{equation*}
                    \Lambda_w : f \mapsto f(w)
                \end{equation*}
                is bounded for each $w \in \Omega$.\label{item:P2}
				
				\myitem[\textbf{P3}] The \emph{shift operators}
                \begin{equation*}
                    S_j : f \mapsto z_j f
                \end{equation*}
                are well-defined for each $1 \leq j \leq d$.\label{item:P3}
			\end{enumerate}
		\end{definition}
		
		\begin{remark}\label{rem:props.P1-P3}
			Let us mention a few words about the properties \ref{item:P1}-\ref{item:P3}.
			\begin{enumerate}
				\item Depending on what the norm is, \ref{item:P2} -- along with Fatou's lemma -- typically guarantees that convergence in $\mathcal{X}$ implies uniform convergence on compact subsets of $\Omega$.
				
				\item Note that, using the closed graph theorem, if \ref{item:P2} and \ref{item:P3} hold then $S_j \in \mathcal{L}(\mathcal{X})$ for each $1 \leq j \leq d$.
				
				\item \ref{item:P1} is used as a proxy condition to ensure that there is a cyclic vector with respect to the notion of cyclicity we work with, as we shall see shortly.
			\end{enumerate}
		\end{remark}
		
		\subsubsection{Shift-cyclic functions} Let $\mathcal{X}$ be a Banach space of analytic functions on some open set $\Omega \subset \mathbb{C}^d$ as above. Let $S = (S_1, \dots, S_d)$ be the $d$-tuple of shift operators, and consider the semigroup $(\mathbb{Z}_+^d, +)$ with coordinate-wise addition as the semigroup operation. For any $\alpha = (\alpha_1, \dots, \alpha_d) \in \mathbb{Z}_+^d$, we write
		\begin{equation*}
			S^\alpha = S_1^{\alpha_1} S_2^{\alpha_2}\dots S_d^{\alpha_d}.
		\end{equation*}
		Now, note that
		\begin{equation*}
			\sigma = \left\{ S^\alpha : \alpha \in \mathbb{Z}_+^d \right\}
		\end{equation*}
		is $\mathbb{Z}_+^d$-compatible. Also, \ref{item:P1} guarantees that the \emph{constant function $1$} is $\sigma$-cyclic. We refer to $\sigma$-cyclicity as \emph{shift-cyclicity} and refer to $\sigma$-cyclic functions as \emph{shift-cyclic functions}. Moreover, we write $S[f]$ instead of $\sigma[f]$, borrowing the notation from the $1$-variable case.
		
		We end this section with a few properties of shift-cyclic functions that will be important for later discussion. Our first result is immediate from the definitions.	
		\begin{proposition}\label{prop:equiv.def.of.shift.cyc}
			
			For any $f \in \mathcal{X}$, we have
			\begin{equation*}
				S[f] = \overline{\operatorname{span}} \left\{ z^\alpha f : \alpha \in \mathbb{Z}_+^d \right\} = \overline{\{Pf : P \in \mathcal{P}_d\}}.
			\end{equation*}
			
			Moreover, the following are equivalent.
			
			\begin{enumerate}
				\item $f$ is shift-cyclic.
				
				\item There exists a sequence of polynomials $\{P_n\}_{n \in \mathbb{N}}$ such that
				\begin{equation}\label{eqn:1-P_nf.goes.to.0}
					\lim_{n \to \infty} \|1 - P_n f\|_\mathcal{X} = 0.
				\end{equation}
				
				\item There exists a shift-cyclic $g \in S[f]$.
			\end{enumerate}
		\end{proposition}
		
		\eqref{eqn:1-P_nf.goes.to.0} combined with \ref{item:P2} gives us a necessary condition for cyclicity.
		
		\begin{corollary}\label{cor:poly.necc.cond.for.cyc}
			If $f \in \mathcal{X}$ is shift-cyclic then $f(w) \neq 0 \foral w \in \Omega$.
		\end{corollary}
		
		\eqref{eqn:1-P_nf.goes.to.0} also implies that a shift-cyclic function $f$ is `almost invertible'. It is important to note, however, that having $1/f \in \mathcal{X}$ is not sufficient for the shift-cyclicity of $f$ (see \cite[Theorem 1.4]{BH97}). We shall study about this algebraic nature of shift-cyclicity in Section \ref{subsubsec:alg.prop.mult-cyc.func} below. Let us also mention another useful property that helps us determine the shift-cyclicity of polynomials.
		
		\begin{proposition}\label{prop:poly.cyc.iff.irred.fac.cyc}
			Let $P \in \mathcal{P}_d$ and $f \in \mathcal{X}$. Then, $Pf$ is shift-cyclic if and only if both $P$ and $f$ are shift-cyclic. Consequently, a polynomial is shift-cyclic if and only if all of its irreducible factors are shift-cyclic.
			
		\end{proposition}
		
		\begin{proof}
			Clearly, $Pf \in S[f]$ and hence, $S[Pf] \subseteq S[f]$ always holds. Since $f$ can be approximated by polynomials (\ref{item:P1}) and multiplication by $P$ is continuous (\ref{item:P3}), we also have $Pf \in S[P]$ and hence, $S[Pf] \subseteq S[P]$. Thus,
			\begin{equation*}
				S[Pf] \subseteq S[P] \cap S[f]
			\end{equation*}
			and we get that if $Pf$ is shift-cyclic then both $P$ and $f$ are shift-cyclic as well.
			
			Conversely, suppose $f$ is shift-cyclic and let $\left\{ P_n \right\}_{n \in \mathbb{N}}$ be as in \eqref{eqn:1-P_nf.goes.to.0}. As noted above, multiplication by $P$ is a continuous operator and hence, $P \in S[Pf]$ since
			\begin{equation*}
				\lim_{n \to \infty} \| P - P_n Pf \|_\mathcal{X} = \lim_{n \to \infty} \| P (1 - P_n f) \|_\mathcal{X} = 0.
			\end{equation*}
			Now, if $P$ is shift-cyclic, then Proposition \ref{prop:equiv.def.of.shift.cyc} $(3)$ shows that $Pf$ is shift-cyclic.
			
			The remaining assertion is immediate from the first part.
		\end{proof}
		
		We now turn our attention to specific Banach spaces of analytic functions and survey what we know about shift-cyclicity in these spaces.


	\section{Hardy spaces}\label{sec:Hardy.sp}
	
	\subsection{On the unit disk \texorpdfstring{$\mathbb{D}$}{}}\label{subsec:Hardy.on.D}
	
	We begin by generalizing the space $H^2(\mathbb{D})$ introduced in Section \ref{subsec:Smirnov-Beurling.thm}. The motivation to study these spaces comes from a natural property of functions in $\operatorname{Hol}(\mathbb{D})$ that was discovered by Hardy \cite{Har15}.
    
    For any $f \in \operatorname{Hol}(\mathbb{D})$, $r < 1$ and $0 < p < \infty$, we define the \emph{radial $p$-means}
	\begin{equation*}
		M_p(r,f) := \left( \int_0^{2\pi} |f(r e^{i \theta})|^p \frac{d \theta}{2 \pi} \right)^{\frac{1}{p}}.
	\end{equation*}
	We also define
	\begin{equation*}
		M_\infty(r,f) := \sup_{\theta \in [0,2\pi)} |f(re^{i \theta})|.
	\end{equation*}
	By the maximum modulus principle, it is clear that $M_\infty(r,f)$ is non-decreasing in $r$ for a fixed $f$, and by the Hadamard three-circles theorem (see \cite[Theorem 12.8]{Rud74}), it follows that $\log M_\infty(r,f)$ is a convex function of $\log r$. Hardy observed that the same is true for $M_p(r,f)$ for all values of $p$\footnote{Hardy mentioned in his paper that the case $p = 1$ was suggested to him by H. Bohr (brother of the physicist N. Bohr) and E. Landau more than a year before he published his result.}.
	
	\begin{theorem}[Theorem I, \cite{Har15}]\label{thm:Hardy.convexity.thm}
		If $f \in \operatorname{Hol}(\mathbb{D})$ and $0 < p \leq \infty$, then
		
		\begin{enumerate}
			\item $M_p(r,f)$ is non-decreasing in $r$ and
			
			\item $\log M_p(r,f)$ is a convex function of $\log r$.
		\end{enumerate}
	\end{theorem}

	The \emph{Hardy space $H^p(\mathbb{D})$} for $0 < p \leq \infty$ is defined as
	\begin{equation*}
		H^p(\mathbb{D}) := \left\{ f \in \operatorname{Hol}(\mathbb{D}) : \lim_{r \to 1^-} M_p(r,f) < \infty \right\}.
	\end{equation*}
	
	\begin{remark}\label{rem:equiv.def.H^2(D)}
		The definition of $H^2(\mathbb{D})$ in \eqref{eqn:def.H^2(D)} is equivalent to the above, since by \emph{Parseval's identity} we have
		\begin{equation*}
			\lim_{r \to 1^-} M_2(r,f)^2 = \lim_{r \to 1^-} \sum_{k \in \mathbb{Z}_+} r^{2k}|c_k|^2 = \langle f,f \rangle_{H^2(\mathbb{D})} \foral f \in H^2(\mathbb{D}).
		\end{equation*}
	\end{remark}
	
	In this section, we describe a few of the many different properties of Hardy space functions. We have cherry-picked only those that will be important to our discussion on shift-cyclicity. The reader is referred to \cite{Dur70, Gar06, Hof07, Nik19, Rud74} to learn more about the Hardy spaces. We start with the most fundamental of these properties. Let us write $\mathbb{T} := \partial \mathbb{D}$ for the \emph{unit circle}.
	
	\begin{theorem}\label{thm:radial.limits.H^p(D)}
		For any $0 < p \leq \infty$ and $f \in H^p(\mathbb{D})$, the limit
        \begin{equation*}
            \lim_{r \to 1^-} f(r\zeta)
        \end{equation*}
        exists for (Lebesgue) a.e. $\zeta \in \mathbb{T}$.
	\end{theorem}
	
	  \noindent \emph{References.} This is essentially a result of Fatou \cite{Fat06}, who proved a general result about radial limits of Poisson integrals of complex measures supported on $\mathbb{T}$. Fatou's theorem gives the case $1 \leq p \leq \infty$ while the rest follows from a result of F. and R. Nevanlinna \cite{NN22}, which states that for any $f \in H^p(\mathbb{D})$ with $0 < p < \infty$, there exist $\varphi, \psi \in H^\infty(\mathbb{D})$ such that $f = \varphi/\psi$.

    \subsubsection{Boundary properties}\label{subsubsec:boundry.prop}
    
	For each $f \in H^p(\mathbb{D})$, we define its \emph{boundary function} a.e. on $\mathbb{T}$ as
	\begin{equation*}
		f^*(\zeta) := \lim_{r \to 1^-} f(r\zeta).
	\end{equation*}
	Boundary functions play an important role in understanding the structure of functions in the Hardy spaces. Let us write $\mathcal{Z}(f)$ for the set of zeros in $\mathbb{D}$ of each $f \in \operatorname{Hol}(\mathbb{D})$ (counting multiplicities) and let $m(0)$ be the multiplicity of the zero at $0$ (if $0 \in \mathcal{Z}(f)$). Riesz \cite{Rie23} gave the following factorization result.
	
	\begin{theorem}\label{thm:Riesz.fact.H^p(D)}
		Given $0 < p \leq \infty$ and $f \in H^p(\mathbb{D})$, consider the Blaschke product
		\begin{equation*}
			B(z) = z^{m(0)} \prod_{w \in \mathcal{Z}(f) \setminus \{0\}} \frac{|w|}{w} \frac{z - w}{1 - \overline{w}z}.
		\end{equation*}
		Then, $|B| < 1$ on $\mathbb{D}$, $f/B \in H^p(\mathbb{D})$, $\mathcal{Z}(f/B) = \emptyset$ and $|f^*| = |(f/B)^*|$ a.e. on $\mathbb{T}$.
	\end{theorem}
	
	Riesz used this factorization property to obtain the connection between a Hardy space function $f$ and its boundary function $f^*$. For each $r < 1$, we write $f_r$ for the \emph{radial dilates of $f$}, i.e.,
    \begin{equation*}
        f_r : w \mapsto f(rw)
    \end{equation*}
    Clearly, $f_r \in H^p(\mathbb{D})$ whenever $f \in H^p(\mathbb{D})$ for some $0 < p \leq \infty$.
	
	\begin{theorem}[Satz 1--3, \cite{Rie23}]\label{thm:mean.conv.prop}
		For any $0 < p < \infty$ and $f \in H^p(\mathbb{D})$, we have
		\begin{equation*}
			\lim_{r \to 1^-} \|f_r^* - f^*\|_{L^p(\mathbb{T})} = 0.
		\end{equation*}
		
		Moreover, $\log |f^*| \in L^1(\mathbb{T})$ unless $f \equiv 0$. Thus, $f \mapsto f^*$ is a bijective map.
	\end{theorem}
	
	It is easy to check that $H^p(\mathbb{D})$ is a vector space for any given $p$. We now define
	\begin{equation*}
		\|f\|_p := \|f^*\|_{L^p(\mathbb{T})} \foral f \in H^p(\mathbb{D}).
	\end{equation*}
    Theorem \ref{thm:mean.conv.prop} shows that
    \begin{equation*}
        \|f\|_p = \lim_{r \to 1^-} M_p(r,f),
    \end{equation*}
    which is how $\|f\|_p$ is typically defined. It can then be checked that $\left\| \cdot \right\|_p$ is a norm on $H^p(\mathbb{D})$ for $1 \leq p \leq \infty$, and that
	\begin{equation*}
		\operatorname{d}(f,g) := \|f - g\|_p^p \foral f, g \in H^p(\mathbb{D})
	\end{equation*}
	is a translation-invariant metric on $H^p(\mathbb{D})$ for $0 < p < 1$. Among other things, Riesz's theorems allow us to obtain the following description of Hardy spaces. See \cite[Section 3.2]{Dur70} for all the details.
	
	\begin{corollary}\label{cor:bdy.prop.H^p(D)}
		For $0 < p < \infty$, consider $\Psi_p : f \mapsto f^*$ from $H^p(\mathbb{D})$ into $L^p(\mathbb{T})$.
		
		\begin{enumerate}
			\item $\Psi_p$ is an isometric isomorphism between $H^p(\mathbb{D})$ and
			\begin{equation*}
				H^p(\mathbb{T}) := \overline{\operatorname{span}}^{L^p(\mathbb{T})} \left\{ e^{i k \theta} : k \in \mathbb{Z}_+ \right\}.
			\end{equation*}
			
			\item $H^p(\mathbb{D})$ is a Banach space for $1 \leq p \leq \infty$ and an F-space for $0 < p < 1$.
			
			\item The set $\mathcal{P}$ of polynomials is dense in $H^p(\mathbb{D})$ for any $0 < p < \infty$.
			
			\item If $1 \leq p \leq \infty$ and $f = \sum_k c_k z^k \in H^p(\mathbb{D})$, then $f^* \in L^p(\mathbb{T})$ corresponds to the Fourier series
            \begin{equation*}
                f^* \sim \sum_{k \in \mathbb{Z}_+} c_k e^{i k \theta}.
            \end{equation*}
		\end{enumerate}
	\end{corollary}
	
	\begin{theorem}\label{thm:H^p(D).is.BSoAF}
		For $1 \leq p < \infty$, $H^p(\mathbb{D})$ is a Banach space of analytic functions.
	\end{theorem}
	
	\begin{proof}
		It follows from Corollary \ref{cor:bdy.prop.H^p(D)} $(2)$ and $(3)$ that $H^p(\mathbb{D}) \subset \operatorname{Hol}(\mathbb{D})$ is a Banach space satisfying \ref{item:P1}. $H^p(\mathbb{D})$ also satisfies \ref{item:P3} trivially.
		
		To check that \ref{item:P2} also holds, first let $p = 2$ and let $w \in \mathbb{D}$ be arbitrary. Now, for
        \begin{equation*}
            k_w(z) := \frac{1}{1 - \overline{w}z} \in H^2(\mathbb{D}),
        \end{equation*}
        it is easy to check that
		\begin{equation}\label{eqn:bdd.eval.H^2(D)}
			|f(w)| = |\langle f, k_w \rangle_{H^2(\mathbb{D})}| \leq \|k_w\|_2 \|f\|_2 \foral f \in H^2(\mathbb{D}).
		\end{equation}
		Thus, \ref{item:P2} is satisfied by $H^2(\mathbb{D})$. For all other $p$, we can replace $f \in H^p(\mathbb{D})$ with $(f/B)^{\frac{p}{2}}$ (where $B$ is as in Theorem \ref{thm:Riesz.fact.H^p(D)}) in \eqref{eqn:bdd.eval.H^2(D)}, and use the fact that $|B| < 1$ on $\mathbb{D}$ and that
        \begin{equation*}
            \|f\|_p = \|f/B\|_p
        \end{equation*}
        to conclude that \ref{item:P2} holds.
	\end{proof}
	
	\subsubsection{Smirnov factorization and shift-cyclicity} While Riesz's theorems are sufficient to understand the Banach space structure of the Hardy spaces, we need to take things one step further and give a concrete description of the non-vanishing function $f/B$ (as in Theorem \ref{thm:Riesz.fact.H^p(D)}) to determine when $f$ is shift-cyclic.
	
	\begin{definition}\label{def:inn.out.func.H^2(D)}
		Let $f \in \operatorname{Hol}(\mathbb{D})$.
		\begin{enumerate}
			\item $f$ is said to be \emph{inner}\footnote{Beurling coined the terms inner/outer in \cite{Beu49}. Therefore, the terminology `Beurling inner/outer' is quite common in the literature.} if $|f| < 1$ on $\mathbb{D}$ and $|f^*| = 1$ a.e. on $\mathbb{T}$.
			
			\item $f$ is said to be \emph{singular inner} if it is a non-vanishing inner function.
			
			\item $f$ is said to be \emph{outer} if $f(0) \neq 0$, $\log |f^*| \in L^1(\mathbb{T})$, and
			\begin{equation}\label{eqn:outer.func.H^p}
				\log |f(0)| = \int_0^{2\pi} \log |f^*(e^{i \theta})| \frac{d \theta}{2 \pi}.
			\end{equation}
		\end{enumerate}
	\end{definition}
	
	It is important to note that every Blaschke product is inner, and every polynomial that is non-vanishing on $\mathbb{D}$ is outer. Also, every inner function evidently lies in $H^\infty(\mathbb{D})$, and every outer function is non-vanishing on $\mathbb{D}$\footnote{$f$ is outer $\Leftrightarrow$ $f(z) = \exp ( \int_\mathbb{T} \frac{w+z}{w-z} |\varphi(w)| \frac{d w}{2 \pi} )$, where $\varphi \in L^1(\mathbb{T})$ (see \cite[Section 4.4.3]{Rud69}).}. Smirnov \cite{Smi29} obtained a factorization property using Fatou's theorem \cite{Fat06} and Jensen's formula \cite{Jen99} (see \cite[Chapter 2]{Dur70} or \cite[Section 2.6]{Nik19} for a proof).
	
	\begin{theorem}\label{thm:Smirnov.fac.H^p(D)}
		For every $f \in H^p(\mathbb{D})$ with $0 < p \leq \infty$, there exists a unique singular inner function $f_\sigma$ and a unique outer function $f_\mu \in H^p(\mathbb{D})$ such that
		\begin{equation*}
			\mathbb{R} \ni f_\sigma(0) > 0 \AND f = B f_\sigma f_\mu,
		\end{equation*}
		where $B$ is the Blaschke product corresponding to $\mathcal{Z}(f)$.
	\end{theorem}
	
	Smirnov \cite{Smi32} and Beurling \cite{Beu49} used this factorization property to show that $f \in H^2(\mathbb{D})$ is shift-cyclic if and only if it is outer. In fact, Beurling showed that if we consider the \emph{inner part of $f$}, i.e., $f_{inn} := B f_\sigma$, then
	\begin{equation}\label{eqn:Smirnov-Beurling.thm.inner.repn}
		S[f] = f_{inn} H^2(\mathbb{D}) = \left\{ f_{inn} g : g \in H^2(\mathbb{D}) \right\}.
	\end{equation}
	Moreover, every closed shift-invariant subspace $E \subset H^2(\mathbb{D})$ is of the form $\varphi H^2(\mathbb{D})$ where $\varphi$ is the \emph{greatest common divisor} of the inner factors of each $f \in E$ (see \cite[Chapter 6 p. 85]{Hof07}), i.e.,
    \begin{enumerate}
        \item $\varphi$ is inner,

        \item $f_{inn}/\varphi$ is inner for each $f \in E$, and

        \item if $\widetilde{\varphi}$ is another such divisor then $\varphi / \widetilde{\varphi}$ is inner
    \end{enumerate}
    Essentially, Beurling shows that the orthogonal projection
    \begin{equation*}
        \operatorname{Proj}_{S[f]} (z^{m(0)})
    \end{equation*}
    can be modified to be an inner function, which then turns out to be $f_{inn}$. Here, $m(0)$ is the multiplicity of the zero of $f$ at $0$.
    
    Beurling's approach extends to every $1 \leq p < \infty$ using $L^p$-duality and the fact that $H^p(\mathbb{D})$ can be identified as a closed subspace of $L^p(\mathbb{T})$ (see \cite[Theorem 7.4]{Dur70}). Lastly, note that we can also make sense of the shift-invariant subspace $S[f]$ for any $f \in H^p(\mathbb{D})$ with $0 < p < 1$ by declaring it as the closure of polynomial multiples of $f$. Gamelin \cite{Gam66} then showed that the Smirnov-Beurling theorem extends to the case $0 < p < 1$ as well. We summarize these results in the following theorem.
	
	\begin{theorem}[Shift-cyclic functions in Hardy spaces]\label{thm:cyc.func.Hardy.sp}
		If $f \in H^p(\mathbb{D})$ with $0 < p < \infty$, then
        \begin{equation*}
            S[f] = H^p(\mathbb{D})
        \end{equation*}
        if and only if $f$ is outer. Also, every closed shift-invariant subspace is of the form $\varphi H^p(\mathbb{D})$ for some inner $\varphi$.
	\end{theorem}
	
	This theorem was generalized by Lax \cite{Lax59} and Halmos \cite{Hal61} to the case of vector-valued Hardy spaces. Several other generalizations and modifications of this theorem exist that this article is too short to contain.
	
	\subsubsection{Properties of shift-cyclic functions}\label{subsubsec:prop.shift-cyc.func}
	
	It will be evident from the upcoming sections that Theorem \ref{thm:cyc.func.Hardy.sp} is indeed quite special, and a complete characterization of shift-cyclic functions seems out of reach in most other spaces. Nevertheless, we can extend several different properties of shift-cyclic functions in the Hardy spaces to general Banach spaces of analytic functions. We end this section by listing a few of these properties below to motivate later discussion.
	
	\begin{enumerate}
		\myitem[\textbf{SCP}]\label{item:SCP} Note that $f(z) = z - w$ satisfies \eqref{eqn:outer.func.H^p} exactly when $w \not\in \mathbb{D}$ (see \cite[Lemma 15.17]{Rud74}). If we combine Theorem \ref{thm:cyc.func.Hardy.sp} and Proposition \ref{prop:poly.cyc.iff.irred.fac.cyc}, we get that $P$ is shift-cyclic in any of the Hardy spaces if and only if it is non-vanishing on $\mathbb{D}$. For a proof that does not use Theorem \ref{thm:cyc.func.Hardy.sp}, see \cite[Theorem 5]{GNN70}.
		
		\myitem[\textbf{MC}]\label{item:MC} It is not difficult to show that $H^\infty(\mathbb{D})$ is the \emph{multiplier algebra} of $H^p(\mathbb{D})$ for each $0 < p < \infty$, i.e.,
		\begin{equation*}
			H^\infty(\mathbb{D}) = \left\{ f : \mathbb{D} \to \mathbb{C} : fg \in H^p(\mathbb{D}) \foral g \in H^p(\mathbb{D}) \right\}.
		\end{equation*}
		For any $f \in H^p(\mathbb{D})$ with $0 < p < \infty$, a result of Garnett \cite[Chapter II Theorem 7.4]{Gar06} then shows that
		\begin{equation*}
			S[f] = \overline{f H^\infty(\mathbb{D})} = \overline{\left\{ fg : g \in H^\infty(\mathbb{D}) \right\}}.
		\end{equation*}
				
		\myitem[\textbf{NI}]\label{item:NI} Note that
        \begin{equation*}
            H^p(\mathbb{D}) \subsetneq H^q(\mathbb{D}) \foral 0 < q < p \leq \infty.
        \end{equation*}
        Since the outerness of any given $f \in H^p(\mathbb{D})$ is independent of $p$, Theorem \ref{thm:cyc.func.Hardy.sp} shows that $f$ is shift-cyclic in $H^p(\mathbb{D})$ if and only if it is shift-cyclic in $H^q(\mathbb{D})$ for all/some $0 < q < p$.
		
		\myitem[\textbf{Mult}]\label{item:Mult} Suppose
        \begin{equation*}
            \frac{1}{p} + \frac{1}{q} = \frac{1}{r}
        \end{equation*}
        and let $\Pi_{p,q} : H^p(\mathbb{D}) \times H^q(\mathbb{D}) \to H^r(\mathbb{D})$ be defined as
		\begin{equation*}
			\Pi_{p,q}(f,g) = fg \foral f \in H^p(\mathbb{D}) \AND g \in H^q(\mathbb{D}).
		\end{equation*}
		$\Pi_{p,q}$ is onto, since, for any given $h = B \widetilde{h} \in H^r(\mathbb{D})$ (see Theorem \ref{thm:Riesz.fact.H^p(D)}), let
		\begin{equation*}
			f = B \widetilde{h}^\frac{r}{p} \in H^p(\mathbb{D}) \AND g = \widetilde{h}^\frac{r}{q} \in H^q(\mathbb{D}).
		\end{equation*}
		Moreover, $\Pi_{p,q}$ is bicontinuous, since
		\begin{equation*}
			\|fg\|_r \leq \|f\|_p \|g\|_q \foral f \in H^p(\mathbb{D}) \AND H^q(\mathbb{D}).
		\end{equation*}
		From here, we can conclude that $f$ and $g$ are both shift-cyclic (in $H^p(\mathbb{D})$ and $H^q(\mathbb{D})$ respectively) if and only if $fg$ is shift-cyclic in $H^r(\mathbb{D})$.
		
		\myitem[\textbf{Inv}]\label{item:Inv} A quick application of \ref{item:Mult} shows that if $f \in H^p(\mathbb{D})$ and $1/f \in H^q(\mathbb{D})$ for some $0 < p,q < \infty$, then $f$ is shift-cyclic in $H^p(\mathbb{D})$.
	\end{enumerate}
	
	\ref{item:SCP} refers to the shift-cyclicity of polynomials. It will be a general theme moving forward that the vanishing set of a polynomial plays a crucial role in determining its shift-cyclicity.
	
	\ref{item:MC} shows that shift-cyclicity in the Hardy spaces is equivalent to \emph{multiplier-cyclicity}, i.e., $f \in H^p(\mathbb{D})$ is multiplier-cyclic if
	\begin{equation*}
		\mathcal{M}[f] := \overline{f H^\infty(\mathbb{D})} = H^p(\mathbb{D}).
	\end{equation*}
	Multiplier-cyclicity is a natural substitute for shift-cyclicity in non-analytic function spaces. We shall learn more about this topic in Section \ref{subsec:mult.cyc}.
	
	\ref{item:NI} refers to the `norm independence' and \ref{item:Mult} refers to the `multiplicativity' of shift-cyclic functions in the Hardy spaces. \ref{item:Inv} refers to the fact that `invertibility' (within the family of Hardy spaces) implies shift-cyclicity. We revisit these properties in Section \ref{subsubsec:alg.prop.mult-cyc.func} for general function spaces.
	
	\subsection{On the unit polydisk \texorpdfstring{$\mathbb{D}^d$}{}}\label{subsec:Hardy.on.D^d}
	
	Before venturing out into exotic function spaces in one variable, let us briefly discuss the situation in the several variable analogue of Hardy spaces. These spaces serve as a reminder of how important the factorization results in $H^p(\mathbb{D})$ are to the understanding of shift-cyclic functions.
	
	Fix $d \in \mathbb{N}$ and consider the unit polydisk $\mathbb{D}^d$ (as in Section \ref{subsubsec:anal.func.sp}). We also consider the \emph{unit $d$-torus}
	\begin{equation*}
		\mathbb{T}^d := \bigtimes_{j = 1}^d  \mathbb{T} = \left\{ \zeta = (\zeta_1,\dots,\zeta_d) \in \mathbb{C}^d : |\zeta_j| = 1 \foral 1 \leq j \leq d \right\},
	\end{equation*}
	equipped with the normalized Lebesgue measure $\lambda_d$. For any $\varphi \in L^1(\mathbb{T}^d)$, we use the compact notation
	\begin{equation*}
		\int_{\mathbb{T}^d} \varphi := \int_{\mathbb{T}^d} \varphi(\zeta) d \lambda_d(\zeta) = \int_{[0,2\pi)^d} \varphi(e^{i \theta_1},\dots,e^{i \theta_d}) \frac{d \theta_1}{2 \pi} \dots \frac{d \theta_d}{2 \pi}.
	\end{equation*}
	For a given $f \in \operatorname{Hol}(\mathbb{D}^d)$ and $r < 1$, let $f_r^* : \mathbb{T}^d \to \mathbb{C}$ be the map
	\begin{equation*}
		f_r^*(\zeta) := f(r\zeta_1,\dots,r\zeta_d) \foral \zeta = (\zeta_1,\dots,\zeta_d) \in \mathbb{T}^d.
	\end{equation*}
	We now define the radial $p$-means of $f$ for each $r < 1$ and $0 < p \leq \infty$ as before:
	\begin{equation*}
		M_p(r,f) := \left( \int_{\mathbb{T}^d} |f_r^*|^p \right)^{\frac{1}{p}} \foral 0 < p < \infty
	\end{equation*}
	and
	\begin{equation*}
		M_\infty(r,f) := \sup_{\mathbb{T}^d} |f_r^*|.
	\end{equation*}
    
	The \emph{Hardy spaces on $\mathbb{D}^d$} are then similarly defined for each $0 < p \leq \infty$ as
	\begin{equation*}
		H^p(\mathbb{D}^d) := \left\{ f \in \operatorname{Hol}(\mathbb{D}^d) : \lim_{r \to 1^-} M_p(r,f) < \infty \right\}.
	\end{equation*}
	As in Remark \ref{rem:equiv.def.H^2(D)}, $H^2(\mathbb{D}^d)$ is the collection of all $f \in \operatorname{Hol}(\mathbb{D}^d)$ with square-summable power-series coefficients. Also, note that
    \begin{equation*}
        H^p(\mathbb{D}^d) \subsetneq H^q(\mathbb{D}^d) \foral p > q.
    \end{equation*}
    The boundary properties of $H^p(\mathbb{D})$ functions are explained by the results of Fatou \cite{Fat06} and Riesz \cite{Rie23}.
	These results are, in essence, consequences of the symmetries of the Poisson kernel and the Poisson integral representation of Hardy space functions.
	Many of these properties can be generalized to $d > 1$ via \emph{slice functions}, i.e., given any $f \in \operatorname{Hol}(\mathbb{D}^d)$ and $\zeta \in \mathbb{T}^d$, consider
	\begin{equation*}
		f_\zeta (w) := f(w\zeta_1,\dots,w\zeta_d) \foral w \in \mathbb{D}.
	\end{equation*}
	For instance, we can use slice functions to show that the radial limits
    \begin{equation*}
        f^*(\zeta) := \lim_{r \to 1^-} f_r^*(\zeta)
    \end{equation*}
    exist a.e. on $\mathbb{T}^d$ \cite[Theorem 3.3.3]{Rud69}. Bochner \cite{Boc69} used slice functions to extend Riesz's theorem and showed that $\log|f^*| \in L^1(\mathbb{T}^d)$ unless $f \equiv 0$, and
	\begin{equation*}
		\lim_{r \to 1^-} \|f^* - f_r^*\|_{L^p(\mathbb{T}^d)} = 0
	\end{equation*}
	for all $f \in H^p(\mathbb{D}^d)$ with $0 < p < \infty$ (see \cite[Theorem 3.4.3]{Rud69} for a different treatment). It is then not difficult to verify that Corollary \ref{cor:bdy.prop.H^p(D)} and Theorem \ref{thm:H^p(D).is.BSoAF} can be generalized for $d > 1$. Unfortunately, we cannot go too far from here in terms of factorization and shift-cyclicity when $d > 1$.
	
	\subsubsection{Negative results}\label{subsubsec:neg.res}
	
	All examples mentioned in this subsection are non-trivial, as well as technical, so we omit the details to keep our discussion compact. The reader is encouraged to look at the references below for these details. Let us start with the following result of Rudin (see \cite[corollary to Theorem 4.4.2]{Rud69}) which shows that Theorem \ref{thm:cyc.func.Hardy.sp} cannot be generalized as is to the case $d > 1$.
	
	\begin{theorem}\label{thm:eg.inftly.gen.shift-inv.subsp}
		$H^2(\mathbb{D}^2)$ has an invariant subspace that is not finitely generated.
	\end{theorem}
	
	One might wonder if we could provide an appropriate description of those shift-invariant subspaces that are finitely generated. The following result of Miles \cite{Mil73} sheds some light on this. Similar to the one variable case, we let $\mathcal{Z}(f)$ denote the zero set of any $f \in \operatorname{Hol}(\mathbb{D}^d)$ in $\mathbb{D}^d$ counting multiplicities\footnote{The multiplicity of a zero at $w$ is the smallest integer $k$ such that the $k$-th term in the \emph{homogeneous expansion} of $f$ around $w$ does not vanish at $w$ (see \cite[Section 1.1]{Rud69}).}.
	
	\begin{theorem}\label{thm:zero.set.problem.d>1}
		There exists $0 \not\equiv f \in H^p(\mathbb{D}^2)$ such that for each $g \in H^q(\mathbb{D}^2)$ with $p < q \leq \infty$ and
        \begin{equation*}
            \mathcal{Z}(f) \subseteq \mathcal{Z}(g),
        \end{equation*}
        we have $g \equiv 0$.
	\end{theorem}
	
	Now, let $f \in H^p(\mathbb{D}^d)$ for some $0 < p < \infty$ and note that
	\begin{equation*}
		\mathcal{Z}(f) \subseteq \mathcal{Z}(g) \foral g \in S[f].
	\end{equation*}
	Combined with Theorem \ref{thm:zero.set.problem.d>1}, the above inclusion gives the following result. Note that while Miles' result holds when $d = 2$, one can easily generalize it to any $d > 1$ by taking
    \begin{equation*}
        \widetilde{f}(z_1,\dots,z_d) := f(z_1,z_2)
    \end{equation*}
    instead of $f$ in the theorem.
	
	\begin{corollary}\label{cor:no.Beurling.thm.d>1}
		For any $d>1$ and $0 < p < \infty$, there is an $f \in H^p(\mathbb{D}^d)$ such that $S[f]$ contains no function from $H^q(\mathbb{D}^d)$ for any given $p < q \leq \infty$. Thus,
        \begin{equation*}
            S[f] \neq g H^p(\mathbb{D}^d) \text{ for any } g \in H^q(\mathbb{D}^d)
        \end{equation*}
	\end{corollary}
	
	The penultimate nail in the coffin is that there is no real hope for a useful factorization result for $d > 1$ either. Firstly, Theorem \ref{thm:zero.set.problem.d>1} shows that the result of F. and R. Nevanlinna \cite{NN22} (see p. \pageref{thm:radial.limits.H^p(D)}) does not hold for $d > 1$. Next, we have the following result of Rosay \cite{Ros75} (generalizing a result of Rudin \cite{Rud66} and Miles \cite{Mil75}). Recall from \ref{item:Mult} that the Smirnov factorization result implies that every $H^1(\mathbb{D})$ function can be factored into a product of two $H^2(\mathbb{D})$ functions.
	
	\begin{theorem}\label{thm:factrzn.problem.d>1}
		If $d > 1$, then there is an $f \in H^1(\mathbb{D}^d)$ that cannot be factored into a product of two $H^2(\mathbb{D}^d)$ functions.
	\end{theorem}
	
	Finally, if one only cares about shift-cyclicity of a function, then it is natural to ask if there is an analogue of an outer function when $d > 1$. Rudin defines an outer function as a function $f \in H^p(\mathbb{D}^d)$ such that $f(0) \neq 0$ and
	\begin{equation*}
		\log |f(0)| = \int_{\mathbb{T}^d} \log |f^*|,
	\end{equation*}
	which carries over many properties from the one variable case; for instance, the properties \ref{item:NI}, \ref{item:Mult} and \ref{item:Inv} were inspired by outerness anyway, so these are obviously carried over to outer functions when $d > 1$.
 
    In \cite[Theorem 4.4.6]{Rud69}, Rudin shows that every shift-cyclic function in $H^p(\mathbb{D}^d)$ is outer. Thus, one-half of the Smirnov--Beurling holds. However, in \cite[Theorem 4.4.8 (b)]{Rud69}, Rudin then constructs an explicit example of an outer function in $H^2(\mathbb{D}^2)$ that is not shift-cyclic.
	
	\begin{theorem}\label{thm:ex.outer.but.not.shift-cyc}
        The function
        \begin{equation*}
        f(z_1,z_2) = \exp \left( \frac{z_1 + z_2 + 2}{z_1 + z_2 - 2} \right)    
        \end{equation*}
		is outer but not shift-cyclic in $H^2(\mathbb{D}^2)$.
	\end{theorem}
	
	\subsubsection{Positive results}\label{subsubsec:pos.res} We start with a result of Ahern and Clark which shows that all shift-invariant subspaces of $H^2(\mathbb{D}^d)$ with finite codimension are generated by finitely many polynomials (see \cite[Theorem 3 and Corollary 4]{AC70}).
	
	\begin{theorem}\label{thm:fin.codim.inv.subsp}
		A shift-invariant subspace $E \subset H^2(\mathbb{D}^d)$ is finite codimensional if and only if it is the closure of a finite codimensional polynomial ideal $\mathcal{I} \triangleleft \mathcal{P}_d$. Moreover,
        \begin{equation*}
            \dim(H^2(\mathbb{D}^d)/E) = \dim(\mathcal{P}_d/\mathcal{I}).
        \end{equation*}
		
		Consequently, if $k < d$ and $f_1, \dots, f_k \in H^2(\mathbb{D}^d)$, then
		\begin{equation*}
			S[f_1,\dots,f_k] := \overline{\operatorname{span}}_{1 \leq j \leq k} S[f_j]
		\end{equation*}
		is either all of $H^2(\mathbb{D}^d)$ or infinite codimensional.
	\end{theorem}
	
	Finite codimensionality of polynomial ideals is fairly well-understood, both algebraically and computationally (see \cite[Section 3.7]{KR00}). In fact, the codimension of such an ideal is precisely the number of common zeros (in $\mathbb{D}^d$) of polynomials in the ideal. Also, the general approach of Ahern and Clark has been expanded upon in several different ways. The text of Chen and Guo \cite{CG03} is an in-depth exploration with further results in this direction (also see \cite{DP89}). For certain operator theoretic characterization of shift-invariant subspaces, see \cite{MMSS19, Man88, SSW13}. 
	
	Soltani gave an equivalent condition for shift-cyclicity of an outer function in $H^2(\mathbb{D}^2)$ via \emph{stationary random fields} (see \cite[Theorem 2.18]{Sol84}), however, from a function theoretic point of view, it leaves a lot to be desired. For instance, it is not clear if the properties mentioned in Section \ref{subsubsec:prop.shift-cyc.func} generalize to $d > 1$. The following result of Ginsberg, Neuwirth and Newman extends \ref{item:SCP} to $d > 1$ (see \cite[Theorem 5]{GNN70})\footnote{The original statement of Theorem 5 in \cite{GNN70} is only valid for $p = 2$ but the argument generalizes easily to all the other values of $p$.}.
	
	\begin{theorem}\label{thm:SCP.Hardy.d>1}
		For any $d \in \mathbb{N}$ and $0 < p < \infty$, a polynomial $P \in \mathcal{P}_d$ is shift-cyclic if and only if it is non-vanishing on $\mathbb{D}^d$.
	\end{theorem}
	
	It is known that $H^\infty(\mathbb{D}^d)$ is the multiplier algebra of $H^p(\mathbb{D}^d)$ for $0 < p < \infty$ (see \cite[Section 2.1 (8)]{Nik12}), i.e.,
    \begin{equation*}
        H^\infty(\mathbb{D}^d) = \left\{ f : \mathbb{D}^d \to \mathbb{C} : fg \in H^p(\mathbb{D}^d) \foral g \in H^p(\mathbb{D}^d) \right\}
    \end{equation*}
    Moreover, \ref{item:MC} extends to $d > 1$ (see \cite[Section 2.1 (9)]{Nik12}), however, we do not know if \ref{item:NI}, \ref{item:Mult} or \ref{item:Inv} generalize to $d > 1$. The following result, which is a slightly modified version of a result of Nikolski (see \cite[Theorem 3.3 (3)]{Nik12}), shows that \ref{item:Inv} `almost' holds when $d > 1$.
	
	\begin{theorem}\label{thm:inv.almost.true.Hardy.d>1}
		If, for some $d > 1$ and $0 < p,q < \infty$, we have $f \in H^p(\mathbb{D}^d)$ and $1/f \in H^q(\mathbb{D}^d)$, then $f$ is shift-cyclic in $H^r(\mathbb{D}^d)$ for all $0 < r < p$.
	\end{theorem}
	
	Since we do not know if \ref{item:NI} holds for $d > 1$, we cannot conclude that $f$ above is shift-cyclic in $H^p(\mathbb{D}^d)$. In fact, if we can find an $f \in H^p(\mathbb{D}^d)$ such that $1/f \in H^q(\mathbb{D}^d)$ for some $0 < p,q < \infty$ but $f$ is not shift-cyclic in $H^p(\mathbb{D}^d)$, then we can conclude that \ref{item:NI} and \ref{item:Mult} do not generalize as such to $d > 1$ (see \cite[Proposition 2.1.1]{Sam22}). This brings us to some interesting open problems.
	
	\begin{problem}\label{prob:ni.Hardy.d>1}
		Does there exist $f \in H^2(\mathbb{D}^2)$ such that $f$ is shift-cyclic in $H^1(\mathbb{D}^2)$ but $f$ is not shift-cyclic in $H^2(\mathbb{D}^2)$?
	\end{problem}
	
	\begin{problem}\label{prob:mult.Hardy.d>1}
		Does there exist $f, g \in H^2(\mathbb{D}^2)$ such that $f$ and $g$ are both shift-cyclic in $H^2(\mathbb{D}^2)$ and that $fg \in H^2(\mathbb{D}^2)$ but $fg$ is not shift-cyclic in $H^2(\mathbb{D}^2)$?
	\end{problem}
	
	\begin{problem}\label{prob:inv.Hardy.d>1}
		Does there exist an $f \in H^2(\mathbb{D}^2)$ such that $1/f \in H^2(\mathbb{D}^2)$ but $f$ is not shift-cyclic in $H^2(\mathbb{D}^2)$?
	\end{problem}
	
	
	\subsection{On the unit multidisk \texorpdfstring{$\mathbb{D}^\infty_2$}{}}\label{subsec:Hardy.on.multi.D}
	
	Hilbert \cite{Hil09} introduced a theory of holomorphic functions on the \emph{unit multidisk} (also, \emph{Hilbert multidisk}) defined as
	\begin{equation*}
		\mathbb{D}^\infty_2 := \left\{ w = (w_1, w_2, \dots ) : |w_n| < 1 \foral n \in \mathbb{N} \AND w \in \ell^2(\mathbb{N}) \right\}.
	\end{equation*}
	We denote $\mathbb{Z}^\infty_+$ to be the collection of sequences $\alpha = (\alpha_1, \alpha_2, \dots)$ of non-negative integers with finite support, i.e., $\alpha_n = 0$ for all $n \in \mathbb{N}$ large enough. We extend the multinomial notation for a given $\alpha \in \mathbb{Z}^\infty_+$ by writing
    \begin{equation*}
        z^\alpha = z^{\alpha_1} z^{\alpha_2} \dots,
    \end{equation*}
    where the formal product truncates after finitely many steps (once $\alpha_n$'s are $0$). Now, we can introduce formal power-series in infinitely many variables
	\begin{equation*}
		F(z) = \sum_{\alpha \in \mathbb{Z}^\infty_+} c_\alpha z^\alpha,
	\end{equation*}
	and define the Hardy space on $\mathbb{D}^\infty_2$ as
	\begin{equation}\label{eqn:def.Hardy.on.multi.D}
		H^2(\mathbb{D}^\infty_2) := \left\{ F = \sum_{\alpha \in \mathbb{Z}^\infty_+} c_\alpha z^\alpha : \sum_{\alpha \in \mathbb{Z}^\infty_+} |c_\alpha|^2 < \infty \right\}.
	\end{equation}
	It is not difficult to check that any power-series whose coefficients satisfy the given condition in \eqref{eqn:def.Hardy.on.multi.D} is holomorphic in the sense of \cite{Hil09}, and that it converges uniformly and absolutely on compact subsets of $\mathbb{D}^\infty_2$ (see \cite[Lemma 6.6.2]{Nik19}). Now, $H^2(\mathbb{D}^\infty_2)$ turns into a Hilbert space under the inner product borrowed from $\ell^2(\mathbb{Z}^\infty_+)$ by the coefficients. Moreover, it can be checked that $H^2(\mathbb{D}^\infty_2)$ satisfies appropriate modifications of \ref{item:P1}-\ref{item:P3} (see \cite[Lemma 6.6.2]{Nik19}).
	
	Note that shift-cyclicity also makes sense for $H^2(\mathbb{D}^\infty_2)$ since the family
	\begin{equation*}
		\{ S^\alpha := S_1^{\alpha_1} S_2^{\alpha_2} \ldots \in \mathcal{L}(H^2(\mathbb{D}^\infty_2)) : \alpha \in \mathbb{Z}^\infty_+ \}
	\end{equation*}
 	is $\mathbb{Z}^\infty_+$-compatible under coordinate-wise addition. Clearly, $H^2(\mathbb{D}^d)$ is embedded isometrically inside $H^2(\mathbb{D}^\infty_2)$ for each $d \in \mathbb{N}$. Moreover, every shift-cyclic $F$ in $H^2(\mathbb{D}^d)$ is also shift-cyclic in $H^2(\mathbb{D}^\infty_2)$, but there are shift-cyclic functions in $H^2(\mathbb{D}^\infty_2)$ that cannot be realized as functions over any such $\mathbb{D}^d$, e.g.,
 	\begin{equation*}
 		K_w(z) := \prod_{n \in \mathbb{N}} \frac{1}{1 - \overline{w_n} z_n},
 	\end{equation*}
 	where $w \in \mathbb{D}^\infty_2$ is arbitrary (see \cite[Corollary 3.7]{Nik12}).
 	
 	Let $H^\infty(\mathbb{D}^\infty_2)$ be the space of all bounded holomorphic functions on $\mathbb{D}^\infty_2$. It is possible to define $H^p(\mathbb{D}^\infty_2)$ for any $0 < p < \infty$ as well. The curious reader can find more information on this fact and the following sample of results about shift-cyclicity in $H^2(\mathbb{D}^\infty_2)$ in \cite{Nik12} (see also the correction from \cite{Nik18}).
 	
 	\begin{theorem}\label{thm:shift-cyc.in.H^2.of.multidisk}
 		$F$ is shift-cyclic in $H^2(\mathbb{D}^\infty_2)$ if either of the following hold.
 		\begin{enumerate}
 			\item $1/F \in H^\infty(\mathbb{D}^\infty_2)$.
 			
 			\item The real part $\mathfrak{Re}(F)$ of $F$ is non-negative on $\mathbb{D}^\infty_2$.
 			
 			\item $F \in \operatorname{Hol}((1+\epsilon) \mathbb{D}^d)$ for some $\epsilon > 0$, $d \in \mathbb{N}$ and $F$ is non-vanishing on $\mathbb{D}^d$.
 			
 			\item $F \in H^p(\mathbb{D}^\infty_2)$ and $1/F \in H^q(\mathbb{D}^\infty_2)$ for some $p > 2$ and $0 < q \leq \infty$.
 			
 			\item $F = F_1 F_2 F_3 F_4$, where each $F_j$ satisfies the condition (j) above.
 		\end{enumerate}
 	\end{theorem}
 	
 	Other interesting partial results along the same lines were obtained in \cite{DG21}. For partial results on the generalization of Theorem \ref{thm:fin.codim.inv.subsp}, see \cite{DGH18}. Let us now see how shift-cyclicity in $H^2(\mathbb{D}^\infty_2)$ connects to, and, in fact, is equivalent to some other interesting problems as hinted in Example \ref{eg:tau-cyclicity}.
	
	\subsubsection{Function theory junction}\label{subsubsec:func.th.junc}
	
	If we realize every $\varphi \in L^2(0,1)$ as an odd $2$-periodic function on $\mathbb{R}$, we obtain a trigonometric series expansion
	\begin{equation*}
		\varphi(x) = \sum_{n \in \mathbb{N}} a_n (\sqrt{2} \sin(\pi n x)).
	\end{equation*}
	Note that $\{\sqrt{2} \sin(\pi n x) : n \in \mathbb{N}\}$ forms an orthonormal basis of $L^2(0,1)$ under this realization. We define the \emph{dilations} (as in Example \ref{eg:tau-cyclicity} $(3)$)
	\begin{equation*}
		D_n : \sqrt{2} \sin(\pi m x) \mapsto \sqrt{2} \sin(\pi m n x),
	\end{equation*}
	and note that $\vartheta := \{ D_n : n \in \mathbb{N} \}$ is $\mathbb{N}$-compatible.
	
	The problem of determining $\vartheta$-cyclic functions is called the \emph{periodic dilation completeness problem (PDCP)}. For $\vartheta$ as in Example \ref{eg:tau-cyclicity} $(3)$, determining when $\chi_{[0,1]} \in \vartheta[\varphi]$ for a given $\varphi \in L^2(0,\infty)$ is called the \emph{dilation completeness problem (DCP)}. Theorem \ref{thm:NBBD.criterion} shows that solving DCP also solves the Riemann hypothesis, but $\varrho$ in the theorem is not $2$-periodic. The connection between PDCP and the Riemann hypothesis was observed by Noor \cite{Noo19}. Wintner \cite{Win44}, Nyman \cite{Nym50} and Beurling \cite{Beu55} were the first to observe the relationship between PDCP and certain properties of \emph{Dirichlet series}. In modern literature, the works of Hedenmalm, Lindquist and Seip \cite{HLS97}, Bagchi \cite{Bag06}, Nikolski \cite{Nik12, Nik18, Nik19} and Noor \cite{Noo19} have been highly influential.
 
    Let $H^2_0(\mathbb{D})$ be as in Example \ref{eg:tau-cyclicity} $(2)$ and note that $L^2(0,1)$ and $H^2_0(\mathbb{D})$ are unitarily equivalent via
    \begin{equation*}
        \sqrt{2} \sin(\pi n x) \leftrightarrow z^n.
    \end{equation*}
    Under the same correspondence, it is also clear that $\vartheta$-cyclicity is equivalent to $\tau$-cyclicity (again, from Example \ref{eg:tau-cyclicity} $(2)$). Taking things one step further we can show that $\tau$-cyclicity is equivalent to shift-cyclicity in $H^2(\mathbb{D}^\infty_2)$. Indeed, we can use an idea of Bohr \cite{Boh13} for this. Let us order all the prime numbers $p_m$ in increasing order along $m \in \mathbb{N}$. Every $n \in \mathbb{N}$ has a unique prime factorization
    \begin{equation*}
        n = p_1^{\kappa_1 (n)} p_2^{\kappa_2 (n)} \ldots,
    \end{equation*}
    which gives us a bijective map $\kappa : \mathbb{N} \to \mathbb{Z}^\infty_+$ via
	\begin{equation*}
		\kappa : n \mapsto (\kappa_1 (n), \kappa_2 (n), \dots).
	\end{equation*}
	We use this bijective map to obtain a correspondence, called the \emph{Bohr transform},
	\begin{equation*}
		\mathfrak{B} : \sum_{n \in \mathbb{N}} c_n z^n \in H^2_0(\mathbb{D}) \mapsto \sum_{\alpha \in \mathbb{Z}_+^\infty} c_{\kappa^{-1}(\alpha)} z^\alpha \in H^2(\mathbb{D}^\infty_2).
	\end{equation*}
	It is a standard exercise to check that $\mathfrak{B} : H^2_0(\mathbb{D}) \to H^2(\mathbb{D}^\infty_2)$ is a unitary and $\tau$-cyclicity is equivalent to shift-cyclicity in $H^2(\mathbb{D}^\infty_2)$ (see \cite[Section 6.6.3]{Nik19}).
	
	While things are clear from an operator theoretic point of view, notice that point-wise relationships are lost under these unitary equivalences. For instance, Nikolski \cite{Nik12} used a theorem of Green and Tao \cite{GT08} about the existence of arbitrarily long arithmetic progressions among the prime numbers to show that
    \begin{equation*}
        \mathfrak{B}(H^\infty(\mathbb{D})) \not\subset H^\infty(\mathbb{D}^\infty_2)
    \end{equation*}
    (also see \cite[Exercise 6.7.4 (c)]{Nik19}). Moreover, from \ref{item:P2} it is clear that if $F$ is shift-cyclic in $H^2(\mathbb{D}^\infty_2)$ then it is non-vanishing on $\mathbb{D}^\infty_2$, but it is not clear what this property corresponds to under $\mathfrak{B}^{-1}$. We therefore end this subsection with the following open problem.
	
	\begin{problem}\label{prob:necc.condn.tau.cyc}
		Are $f(\mathbb{D})$ and $\mathfrak{B}f(\mathbb{D}^\infty_2)$ related to each other for any given $f \in H^2_0(\mathbb{D})$? In particular, can we classify all $f \in H^2_0(\mathbb{D})$ such that $\mathfrak{B}f$ is non-vanishing on $\mathbb{D}^\infty_2$.
	\end{problem}
	
	
	\section{Dirichlet-type spaces}\label{sec:Dirich-type.sp}
	
	
	\subsection{On the unit disk \texorpdfstring{$\mathbb{D}$}{}}\label{subsec:Dirich.on.d=1}
	
	In this section, we present a situation where shift-cyclicity is mysterious despite functions in the space exhibiting the same factorization and boundary properties as the one variable Hardy spaces, but the reason this time is partly due to functions having a bit more rigidity along the boundary. We follow the notation from \cite{EKMR14} throughout this section.
	
	\subsubsection{The Dirichlet space \texorpdfstring{$\mathcal{D}$}{}}\label{subsubsec:Dirich.sp}
	
	The \emph{Dirichlet space on $\mathbb{D}$} is defined as
	\begin{equation*}
		\mathcal{D} := \left\{ f = \sum_{k \in \mathbb{Z}_+} c_k z^k : \sum_{k \in \mathbb{Z}_+} (k + 1) |c_k|^2 < \infty \right\}.
	\end{equation*}
	The following properties of $\mathcal{D}$ are evident from the definition:
	\begin{enumerate}
		\item $\mathcal{D}$ turns into a Hilbert space with the inner product
		\begin{equation*}
			\left\langle \sum_{k \in \mathbb{Z}_+} c_k z^k, \sum_{k \in \mathbb{Z}_+} d_k z^k \right\rangle_\mathcal{D} := \sum_{k \in \mathbb{Z}_+} (k + 1) c_k \overline{d_k}.
		\end{equation*}
		
		\item $\mathcal{D} \subsetneq H^2(\mathbb{D})$\footnote{In fact, $\mathcal{D} \subsetneq \bigcap_{p < \infty} H^p(\mathbb{D})$, $H^\infty(\mathbb{D}) \not\subset \mathcal{D}$ and $\mathcal{D} \not\subset H^\infty(\mathbb{D})$ (see \cite[Exercise 1.1 (7)]{EKMR14}).}. Thus, functions in $\mathcal{D}$ have radial limits a.e. on $\mathbb{T}$ and exhibit an inner-outer factorization.
		
		\item $\mathcal{D}$ is a Hilbert space of analytic functions on $\mathbb{D}$.
		
		\item An inequality of Hardy (see \cite[the corollary to Theorem 3.15]{Dur70}) shows that if $f' \in H^1(\mathbb{D})$ then $f \in \mathcal{D}$. The function $f(z) = \sum_n \frac{z^n}{n}$ shows that $f$ need not lie in $\mathcal{D}$ even if $f' \in H^p(\mathbb{D})$ for all $p < 1$.
	\end{enumerate}
	
	Things get complicated from here on. For starters, one can check that if $\{w_n : n \in \mathbb{N}\}$ is the zero set of a Blaschke product $B$, then $f(z) = (1-z)^2 B(z)$ is such that $f' \in H^\infty(\mathbb{D})$. Thus, $(4)$ above shows that $\{w_n : n \in \mathbb{N}\}$ is the zero set of some $f \in \mathcal{D}$. Interestingly, infinite Blaschke products do not lie in $\mathcal{D}$. In fact, the only inner functions that lie in $\mathcal{D}$ are finite Blaschke products (this follows from a result of Carleson \cite{Car60}; see also \cite[Theorem 4.1.3]{EKMR14}). Determining all the zero sets of functions in $\mathcal{D}$ is still an open problem.
	
	Next, we note that $H^\infty(\mathbb{D})$ is not the multiplier algebra of $\mathcal{D}$\footnote{$\mathcal{M}(\mathcal{D})$ is not even algebraically isomorphic to $H^\infty(\mathbb{D})$ (see \cite[Corollary 7.4]{DHS15}).}. In fact, the multiplier algebra $\mathcal{M}(\mathcal{D})$ does not seem to have a simple description in terms of integral or functional properties of its members. Since $\mathcal{D}$ does not contain every inner function, we know that $H^\infty(\mathbb{D}) \not\subset \mathcal{D}$ and it is not difficult to check that
    \begin{equation*}
        \mathcal{M}(\mathcal{D}) \subset \mathcal{D} \cap H^\infty(\mathbb{D}),
    \end{equation*}
    but \cite[Theorem 5.1.6]{EKMR14} shows that
    \begin{equation*}
        (\mathcal{D} \cap H^\infty(\mathbb{D})) \setminus \mathcal{M}(\mathcal{D}) \neq \emptyset.
    \end{equation*}
 
    Lastly, every shift-invariant subspace of $\mathcal{D}$ is of the form $S[\varphi]$ for some $\varphi \in \mathcal{M}(\mathcal{D})$\footnote{A corollary of this fact is that the zero sets of members in $\mathcal{D}$ and $\mathcal{M}(\mathcal{D})$ are the same.} (see \cite[Theorem 2 (c)]{RS88}), but we do not have a good description of all such $\varphi$'s. This does tell us, though, that \ref{item:MC} generalizes to $\mathcal{D}$ with $\mathcal{M}(\mathcal{D})$ instead of $H^\infty(\mathbb{D})$. \cite[Theorem 9.1.3]{EKMR14} shows that if $f \in \mathcal{D}$, then
	\begin{equation*}
		S[f] = f_{inn} S[f_{out}] \cap \mathcal{D} = S[f_{out}] \cap f_{inn} H^2(\mathbb{D}).
	\end{equation*}
	Consequently, if $f$ is shift-cyclic in $\mathcal{D}$, then it must be outer. However, just like the Hardy spaces for $d > 1$, there are outer functions that are not shift-cyclic (see \cite[Theorem 9.2.8]{EKMR14}). It turns out that in order to understand shift-cyclicity in $\mathcal{D}$ we have one more necessary condition to consider alongside outerness.
	
	\subsubsection{Logarithmic capacity and shift-cyclicity}\label{subsubsec:log.cap.and.shift-cyc}
	
	Fix a closed set $\Gamma \subset \mathbb{T}$ and let $\mathfrak{M}^+_1(\Gamma)$ be the space of all probability measures on $\Gamma$. The \emph{logarithmic capacity} of $\Gamma$ is defined as
	\begin{equation*}
		\operatorname{cap}(\Gamma) := \left[ \inf_{\mu \in \mathfrak{M}^+_1(\Gamma)} \int_{\Gamma \times \Gamma} \log \left(\frac{2}{|\zeta_1 - \zeta_2|}\right) d\mu(\zeta_1) d\mu(\zeta_2)\right]^{-1}.
	\end{equation*}
	For any arbitrary set $\Gamma \subset \mathbb{T}$, we define its \emph{inner logarithmic capacity} as
	\begin{equation*}
		\operatorname{cap}_*(\Gamma) := \sup_{\substack{V \subset \Gamma \\ \text{compact}}} \operatorname{cap}(V),
	\end{equation*}
	and its \emph{outer logarithmic capacity} as
	\begin{equation*}
		\operatorname{cap}^*(\Gamma) := \inf_{\substack{U \supset \Gamma \\ \text{open}}} \operatorname{cap}_*(U).
	\end{equation*}
	
	A set $\Gamma \subset \mathbb{T}$ is said to be \emph{capacitable} if $\operatorname{cap}_*(\Gamma) = \operatorname{cap}^*(\Gamma)$, in which case we simply write $\operatorname{cap}(\Gamma)$ for either. A classical result of Choquet \cite{Cho54} shows that every Borel subset of $\mathbb{T}$ is capacitable.
	%
	%
	The capacity of a Borel set $\Gamma \subset \mathbb{T}$ is intimately connected to its (normalized) Lebesgue measure $|\Gamma|$. For instance, \cite[Theorem 2.4.5]{EKMR14} shows that if $|\Gamma| > 0$ then
	\begin{equation*}
		\operatorname{cap}(\Gamma) \geq \frac{1}{1 - \log |\Gamma|}.
	\end{equation*}
	It follows that if $\operatorname{cap}(\Gamma) = 0$ then $|\Gamma| = 0$ as well. However, the \emph{Cantor middle thirds set $C_{1/3}$} is such that $|C_{1/3}| = 0$ but $\operatorname{cap}(C_{1/3}) > 0$ (see \cite[Theorem 2.3.5]{EKMR14}). For this reason, the following result of Beurling \cite{Beu40} highlights an important deviation of $\mathcal{D}$ from the Hardy spaces. Recall that a property holds \emph{quasi everywhere (q.e.)} if it holds everywhere except on a set $\Gamma$ with $\operatorname{cap}^*(\Gamma) = 0$.
	
	\begin{theorem}\label{thm:q.e.rad.lim.Dirich.sp}
		Every $f \in \mathcal{D}$ has radial limits $\lim_{r \to 1^-} f(r \zeta)$ for q.e. $\zeta \in \mathbb{T}$.
	\end{theorem}
	
	For a given $f \in \mathcal{D}$, let
	\begin{equation*}
		\mathcal{Z}_\mathbb{T}(f^*) := \{ \zeta \in \mathbb{T} : f^*(\zeta) = 0 \}.
	\end{equation*}
	The following result of Brown and Shileds \cite[Theorem 5]{BS84} shows how logarithmic capacity relates to shift-cyclicity.
	
	\begin{theorem}\label{thm:shift-cyc.in.Dirich.sp.implies.cap=0}
		If $f \in \mathcal{D}$ is such that $\operatorname{cap}^*(\mathcal{Z}_\mathbb{T}(f^*)) > 0$ then $f$ is not shift-cyclic.
	\end{theorem}
	
	As a result of this theorem, we arrive at the famous \emph{Brown--Shields conjecture} (appeared as Question 12 in \cite{BS84}), which is still not resolved.
	
	\begin{problem}\label{prob:Brown-Shields.conj}
		Is it true that an outer $f \in \mathcal{D}$ is shift-cyclic if $\operatorname{cap}^*(\mathcal{Z}_{\mathbb{T}}(f^*)) = 0$?
	\end{problem}
	
	A compilation of sufficient conditions for the shift-cyclicity of a given $f \in \mathcal{D}$ can be found in \cite[Chapter 9]{EKMR14}, but, in general, the question is still open. We end this section with some notable results about shift-cyclicity in $\mathcal{D}$.
	\begin{enumerate}
		\item A polynomial is shift-cyclic in $\mathcal{D}$ if and only if it is non-vanishing on $\mathbb{D}$ (see \cite[Theorem 9.2.1]{EKMR14}).
		
		\item If Problem \ref{prob:Brown-Shields.conj} can be resolved for $f \in \mathcal{M}(\mathcal{D})$, then it can be resolved entirely (see \cite[Corollary 2.5.2]{Sam22}).
	\end{enumerate}
	Let us now introduce an important family of spaces containing $H^2(\mathbb{D})$ and $\mathcal{D}$ before moving to the several variable analogue of $\mathcal{D}$.
	
	\subsubsection{Dirichlet-type spaces on \texorpdfstring{$\mathbb{D}$}{the unit disk}}\label{subsubsec:D_t.d=1}
	
	For $t \in \mathbb{R}$, we define the \emph{Dirichlet-type space}
	\begin{equation*}
		\mathcal{D}_t := \left\{ f(z) = \sum_{k \in \mathbb{Z}_+} c_k z^k : \sum_{k \in \mathbb{Z}_+} (k + 1)^t |c_k|^2 < \infty \right\}.
	\end{equation*}
	Just like the Dirichlet space, it is easy to verify that each $\mathcal{D}_t$ is a Hilbert space of analytic functions on $\mathbb{D}$ with the inner product
	\begin{equation*}
		\left\langle \sum_{k \in \mathbb{Z}_+} c_k z^k, \sum_{k \in \mathbb{Z}_+} d_k z^k \right\rangle_{\mathcal{D}_t} := \sum_{k \in \mathbb{Z}_+} (k+1)^t c_k \overline{d_k}.
	\end{equation*}
	For $t < 0$, we have an equivalent norm for $\mathcal{D}_t$ given by
	\begin{equation*}
		\|f\|_{\mathcal{D}_t} := \int_\mathbb{D} |f(w)|^2 (1 - |w|^2)^{-1-t} dA(w),
	\end{equation*}
	where
    \begin{equation*}
        dA := \frac{1}{\pi} dx dy
    \end{equation*}
    is the normalized area measure on $\mathbb{D}$ (see \cite[Lemma 2]{Tay66}).	Note that for $t = 0$ and $t = 1$, we recover $H^2(\mathbb{D})$ and $\mathcal{D}$ respectively. Another important space in this family is the \emph{Bergman space $\mathcal{D}_{-1} = L^2_a(\mathbb{D})$} which will turn up frequently in Section \ref{sec:complete.Pick.sp}. Also, note that
    \begin{equation*}
        \mathcal{D}_t \subsetneq \mathcal{D}_s \foral t > s.
    \end{equation*}
	
	For $t > 1$, functions in $\mathcal{D}_t$ are continuous up to $\overline{\mathbb{D}}$. In fact, $\mathcal{D}_t$ is an algebra under point-wise multiplication and, as a result, $\mathcal{D}_t$ is its own multiplier algebra. It is then not so difficult to check that $f$ is shift-cyclic in $\mathcal{D}_t$ if and only if it is non-vanishing on $\overline{\mathbb{D}}$ (see \cite[pp. 274]{BS84}). For $t \leq 0$, the above equivalent integral norm shows that the multiplier algebra of $\mathcal{D}_t$ is $H^\infty(\mathbb{D})$, whereas the situation is much more complicated for $0 < t \leq 1$. Also, for $t < 0$, there are singular inner functions that turn out to be shift-cyclic in $\mathcal{D}_t$. Korenblum \cite{Kor81} classified all shift-cyclic inner functions in $L^2_a(\mathbb{D})$. With these properties in mind, we highlight the following results from \cite{BS84} (also see \cite{BCLSS151} for a modern treatment of Dirichlet-type spaces on $\mathbb{D}$).
	\begin{enumerate}
		
		\item \ref{item:MC} holds for $\mathcal{D}_t$ for all $t \in \mathbb{R}$, i.e., for any $t \in \mathbb{R}$ and $f \in \mathcal{D}_t$, we have
		\begin{equation*}
			S[f] = \overline{f \mathcal{M}(\mathcal{D}_t)},
		\end{equation*}
		where
        \begin{equation*}
            \mathcal{M}(\mathcal{D}_t) := \{ f : \mathbb{D} \to \mathbb{C} : fg \in \mathcal{D}_t \foral g \in \mathcal{D}_t \}
        \end{equation*}
        denotes the multiplier algebra of $\mathcal{D}_t$.
		
		\item If $f \in \mathcal{D}_t$ and $\varphi \in \mathcal{M}(\mathcal{D}_t)$ for any $t \in \mathbb{R}$, then $\varphi f$ is shift-cyclic in $\mathcal{D}_t$ if and only if both $\varphi$ and $f$ are shift-cyclic in $\mathcal{D}_t$.
		
		\item If $f \in H^2(\mathbb{D})$\footnote{This result is valid for all $f \in H^1(\mathbb{D})$ as well, since one can check that $H^1(\mathbb{D}) \subset L^2_a(\mathbb{D})$.} then it is shift-cyclic in $L^2_a(\mathbb{D})$ if and only if its inner part is shift-cyclic in $L^2_a(\mathbb{D})$ (which are determined in \cite{Kor81}).
		
		\item If $f$ is shift cyclic in $\mathcal{D}_{t_0}$ for some $t_0 \in \mathbb{R}$, then it is shift-cyclic in $\mathcal{D}_t$ for all $t > t_0$. However, (3) shows that there is an $f$ that is shift-cyclic in $\mathcal{D}_{-1}$ but not shift-cyclic in $\mathcal{D}_0$. Thus, \ref{item:NI} does not generalize to the family $\{\mathcal{D}_t : t \in \mathbb{R}\}$.
	\end{enumerate}
	
	Lastly, we remark that a polynomial is shift-cyclic in $\mathcal{D}_t$ with $t \leq 1$ if and only if it is non-vanishing on $\mathbb{D}$. This follows from the first part of (4) above, combined with \cite[Theorem 9.2.1]{EKMR14}. As noted before, a polynomial is shift-cyclic in $\mathcal{D}_t$ with $t > 1$ if and only if it is non-vanishing on $\overline{\mathbb{D}}$. We saw with the Hardy spaces over $\mathbb{D}^d$ that determining the shift-cyclic polynomials is tractable for all $d \in \mathbb{N}$. In the following two subsections, we introduce generalizations of the Dirichlet-type spaces over $\mathbb{D}^d$ and $\mathbb{B}_d$, and note that the problem of determining shift-cyclic polynomials when $d = 2$ is resolved, but it remains open for $d > 2$.
	
	\subsection{On the unit polydisk \texorpdfstring{$\mathbb{D}^d$}{}}\label{subsec:Dirich.on.D^d}
	
	For $d \in \mathbb{N}$ and $t \in \mathbb{R}$, we define the \emph{Dirichlet-type spaces on $\mathbb{D}^d$} as
	\begin{equation*}
		\mathcal{D}_t(\mathbb{D}^d) := \left\{ f(z) = \sum_{\alpha \in \mathbb{Z}_+^d} c_\alpha z^\alpha : \sum_{\alpha \in \mathbb{Z}_+^d} \left( \prod_{j = 1}^d (\alpha_j + 1)^t \right) |c_\alpha|^2  < \infty \right\}.
	\end{equation*}
	Just like their one variable counter-part, these are all Hilbert spaces of analytic functions on $\mathbb{D}^d$, and we have the natural inclusion
    \begin{equation*}
        \mathcal{D}_t(\mathbb{D}^d) \subsetneq \mathcal{D}_s(\mathbb{D}^d) \foral t > s.
    \end{equation*}
    Moreover, we obtain the Bergman, Hardy and Dirichlet spaces on $\mathbb{D}^d$ by taking $t = -1, 0$ and $1$ respectively. For $t > 1$, $\mathcal{D}_t(\mathbb{D}^d)$ is a subalgebra of functions in $\operatorname{Hol}(\mathbb{D}^d)$ that are continuous up to $\overline{\mathbb{D}^d}$ and thus, is its own multiplier algebra. Also, just like in the one variable case, a function $f$ is shift-cyclic in $\mathcal{D}_t(\mathbb{D}^d)$ with $t > 1$ if and only if it is non-vanishing on $\overline{\mathbb{D}^d}$ (see \cite[Section 1.1]{BKKLSS16}).
	
	While we can extend some of the general properties mentioned in the previous section to all $d > 1$, one only needs to look at the polynomials to realize that classifying all shift-cyclic functions is further beyond the scope for these spaces than one might anticipate. In fact, we only have a complete classification of shift-cyclic polynomials for $d = 2$. Let $\mathcal{Z}_\Gamma(P)$ denote the set of zeros of a polynomial $P \in \mathcal{P}_d$ on some $\Gamma \subseteq \mathbb{C}^d$. The following is a joint result of B{\'e}n{\'e}teau, Condori, Knese, Kosi{\'n}ski, Liaw, Seco and Sola \cite{BCLSS152, BKKLSS16}\footnote{See also \cite{KKRS19} for a similar result for the \emph{anisotropic Dirichlet-type spaces on $\mathbb{D}^2$}.}.
	
	\begin{theorem}\label{thm:cyc.poly.Dirich-type.d=2}
		If $P$ is an irreducible polynomial with $\mathcal{Z}_{\mathbb{D}^2}(P) = \emptyset$, then it is shift-cyclic in $\mathcal{D}_t(\mathbb{D}^2)$ for
		\begin{enumerate}
			\item $t \leq 1/2$,
			
			\item $1/2 < t \leq 1$ if and only if either
                \begin{enumerate}
                    \item $\mathcal{Z}_{\mathbb{T}^2}(P)$ is a finite set,

                    \item $P = z_1 - \zeta$ for some $\zeta \in \mathbb{C} \setminus \mathbb{D}$, or

                    \item $P = z_2 - \zeta$ for some $\zeta \in \mathbb{C} \setminus \mathbb{D}$.
                \end{enumerate}
			
			\item $t > 1$ if and only if $\mathcal{Z}_{\mathbb{T}^2}(P) = \emptyset$.
		\end{enumerate}
	\end{theorem}
	
	There are a few things to unpack here. First, the proof of this theorem relies on something called a \emph{determinantal representation} for polynomials that are non-vanishing on $\mathbb{D}^2$, which may not always exist for $d > 2$ \cite{Kne101, Kne102}. Next, this result clearly indicates how boundary zeros are important to study shift-cyclicity of polynomials. In particular, the authors use the notion of \emph{Riesz $t$-capacity} $\operatorname{cap}_t$ to determine the non-cyclicity of a polynomial $P$ in $\mathcal{D}_t(\mathbb{D}^2)$ for large enough $t$, based on how large $\operatorname{cap}_t(\mathcal{Z}_{\mathbb{T}^2}(P))$ is.
	
	Arguably, the most interesting aspect of Theorem \ref{thm:cyc.poly.Dirich-type.d=2} is the case $1/2 < t \leq 1$, where finiteness of the boundary zero set is crucial to determining whether a polynomial is shift-cyclic or not. The authors achieve this via the use of tools from \emph{semi-analytic geometry}. More specifically, they use an inequality of {\L}ojasiewicz \cite{Loj59}, which states that if $f$ is a non-zero real analytic function on some open set $U \subset \mathbb{R}^N$ with $N \in \mathbb{N}$ whose zero set $\mathcal{Z}_U(f)$ in $U$ is non-empty, then, for every compact $V \subset U$, there are constants $C_V > 0$ and $n_V \in \mathbb{N}$ such that
	\begin{equation*}
		|f(x)| \geq C_V \operatorname{dist}(x, \mathcal{Z}_U(f))^{n_V} \foral x \in V.
	\end{equation*}
	Applying this inequality to $|P|^2$, where $P$ is a non-vanishing polynomial on $\mathbb{D}^2$ with finitely many zeros in $\mathbb{T}^2$, allows us to compare $P$ with other polynomials that are shift-cyclic (see \cite[Lemma 3.3]{BKKLSS16}). This further enables us to determine the shift-cyclicity of such $P$ in $\mathcal{D}_t(\mathbb{D}^2)$ when $t \leq 1$.
    
    Bergqvist \cite{Ber18} extended some of these ideas to the $d > 2$ case and presented several conditions on polynomials that guarantee shift-cyclicity as well as non-cyclicity of polynomials in $\mathcal{D}_t(\mathbb{D}^d)$. However, the problem of determining shift-cyclicity of a general polynomial in $\mathcal{D}_t(\mathbb{D}^d)$ remains unresolved.
	
	\subsection{On the unit ball \texorpdfstring{$\mathbb{B}_d$}{}}\label{subsec:Dirich.on.B_d}

    In this subsection, we note that the {\L}ojasiewicz inequality and the notion of Riesz $t$-capacity helps generalize this result to a similar family of spaces over $\mathbb{B}_2$, and also allows us to go beyond the $d = 2$ case to obtain some partial results on shift-cyclicity of polynomials.
    
    Let $d \in \mathbb{N}$ be fixed. For each $\alpha \in \mathbb{Z}_+^d$, we write
	\begin{equation*}
		|\alpha| = \sum_{j = 1}^d \alpha_j, \AND \alpha! = \prod_{j = 1}^{d} \alpha_j !.
	\end{equation*}
	Now, for each $t \in \mathbb{R}$, we define the \emph{Dirichlet-type spaces on $\mathbb{B}_d$} as
	\begin{equation*}
		\mathcal{D}_t(\mathbb{B}_d) := \left\{ f(z) = \sum_{\alpha \in \mathbb{Z}_+^d} c_\alpha z^\alpha : \sum_{\alpha \in \mathbb{Z}_+^d} (d + |k|)^t \frac{(d-1)!\alpha!}{(d - 1 + |\alpha|)!} |c_\alpha|^2 < \infty \right\}.
	\end{equation*}
	Per usual, these can be verified to be Hilbert spaces of analytic functions on $\mathbb{B}_d$ such that
    \begin{equation*}
        \mathcal{D}_{t_1}(\mathbb{B}_d) \subsetneq \mathcal{D}_{t_2}(\mathbb{B}_d) \foral t_1 > t_2.
    \end{equation*}
    Some important spaces within this family are the Bergman, Hardy, \emph{Drury--Arveson} and Dirichlet spaces on $\mathbb{B}_d$, corresponding to $t = -1, \, 0, \, d-1$ and $d$ respectively. Moreover, $\mathcal{D}_t(\mathbb{B}_d)$ is a subalgebra of functions in $\operatorname{Hol}(\mathbb{B}_d)$ that are continuous up to $\overline{\mathbb{B}_d}$ whenever $t > d$. For more details on function theory over $\mathbb{B}_d$, see \cite{Rud80} and \cite{Zhu05}.
	
	Kosinski and Vavitsas \cite{KV23} used the {\L}ojasiewicz inequality and some more tools from semi-analytic geometry to obtain the following classification of shift-cyclic polynomials for $d = 2$. Let $\mathbb{S}_d$ be the boundary of $\mathbb{B}_d$ for $d \in \mathbb{N}$, i.e.,
	\begin{equation*}
		\mathbb{S}_d := \left\{ w \in \mathbb{C}^d : \sum_{j = 1}^d |w_j|^2 = 1 \right\}.
	\end{equation*}
	
	\begin{theorem}\label{thm:cyc.poly.Dirich-typ.sp.ball.d=2}
		If $P \in \mathcal{P}_2$ with $\mathcal{Z}_{\mathbb{B}_2}(P) = \emptyset$, then it is shift-cyclic in $\mathcal{D}_t(\mathbb{B}_2)$ for
		\begin{enumerate}
			\item $t \leq 3/2$,
			
			\item $3/2 < t \leq 2$ if and only if $\mathcal{Z}_{\mathbb{S}_2}(P)$ is a finite set, and
			
			\item $t > 2$ if and only if $\mathcal{Z}_{\mathbb{S}_2}(P) = \emptyset$.
		\end{enumerate}
	\end{theorem}
	
	Here, the concept of Riesz $t$-capacity is also relevant. In fact, Vavitsas \cite[Theorem 16]{Vav23} shows that the capacity of the boundary zero set of a shift-cyclic polynomial on $\mathbb{B}_d$ must be $0$. In a recent preprint, Vavitsas and Zarvalis \cite{VZ24} use similar ideas and obtain the following result.
	
	\begin{theorem}\label{thm:non-cyc.poly.Dirich-type.ball}
		Let $P \in \mathcal{P}_d$ be non-vanishing on $\mathbb{B}_d$ and suppose $\mathcal{Z}_{\mathbb{S}_d}(P)$ contains a real submanifold (of $\mathbb{R}^{2d}$) of dimension $k$ with $1 \leq k \leq d - 1$, but no submanifold of higher dimension. Then, $P$ is not shift-cyclic in $\mathcal{D}_t(\mathbb{B}_d)$ for $t > \frac{2d - (k - 1)}{2}$.
	\end{theorem}
	
	The authors point out that if $\mathcal{Z}_{\mathbb{S}_d}(P)$ is a finite set (in other words, $k = 0$ in the statement of the above theorem), then one can argue as in the proof of Theorem \ref{thm:cyc.poly.Dirich-typ.sp.ball.d=2} to determine its shift-cyclicity in each $\mathcal{D}_t(\mathbb{B}_d)$. Thus, they conjecture that the following question has a positive answer, which then classifies all shift-cyclic polynomials in the Dirichlet-type spaces over $\mathbb{B}_d$.
	
	\begin{problem}\label{prob:cyc.poly.Dirich-typ.sp.ball}
		Let $P \in \mathcal{P}_d$ be as in the hypothesis of Theorem \ref{thm:non-cyc.poly.Dirich-type.ball}. Then, is it true that $P$ is shift-cyclic in $\mathcal{D}_t(\mathbb{B}_d)$ if and only if $t \leq \frac{2d - (k - 1)}{2}$?
	\end{problem}
	
	
	
	\section{Complete Pick spaces}\label{sec:complete.Pick.sp}

    For the remaining discussion, we shall shift gears and turn to an operator theoretic approach of studying shift-cyclicity of functions. As a result, we are able to say more about multiplier-cyclicity for function spaces that may not necessarily be spaces of analytic functions (or function spaces over domains in $\mathbb{C}^d$ for any $d \in \mathbb{N}$). First, we need to set up some important notation.
	
	\subsection{Reproducing kernel Hilbert spaces}\label{subsec:CP.prop}
	
	Let $X$ be a set and let $\mathcal{F}(X,\mathbb{C})$ be the space of all complex-valued functions on $X$. A space $\mathcal{H} \subset \mathcal{F}(X,\mathbb{C})$ is said to be a \emph{reproducing kernel Hilbert space (RKHS)} if it is a Hilbert space (with inner product $\langle \cdot, \cdot \rangle_\mathcal{H}$) such that the point-evaluations $\Lambda_x : f \mapsto f(x)$ are bounded linear functionals on $\mathcal{H}$. By the Riesz representation theorem, it then follows that there exist $K_x \in \mathcal{H}$ for each $x \in X$ such that
	\begin{equation*}
		f(x) = \langle f, K_x \rangle_\mathcal{H} \foral f \in \mathcal{H}.
	\end{equation*}
	The map $K : X \times X \to \mathbb{C}$ given by
	\begin{equation*}
		K(x,y) := \langle K_y, K_x \rangle_\mathcal{H} \foral x,y \in X
	\end{equation*}
	is called the \emph{reproducing kernel} (or, simply, \emph{kernel}) of $\mathcal{H}$, and the functions $K_x$ are called the \emph{kernel functions} for each $x \in X$. Using a standard orthogonality/separation argument, it is not difficult to show that $\mathcal{H} = \mathcal{H}(\mathcal{K})$, where
	\begin{equation*}
		\mathcal{H}(K) := \overline{\operatorname{span}} \{ K_x : x \in X \},
	\end{equation*}
	Indeed, if $f \in \mathcal{H} \ominus \mathcal{H}(K)$ then $f \equiv 0$, because
    \begin{equation*}
        \langle f, K_x\rangle_\mathcal{H} = 0 \foral x \in X.
    \end{equation*}
	
	\begin{example}\label{exam:RKHS.examples}
		\cite[Proposition 2.18]{AM02} shows that the reproducing kernel $K$ of a RKHS $\mathcal{H}$ can be obtained from any orthonormal basis $\{ f_\xi : \xi \in \Xi \}$ via
		\begin{equation*}
			K(x,y) = \sum_{\xi \in \Xi} f_\xi(x)\overline{f_\xi(y)}.
		\end{equation*}
		We use this fact to obtain the reproducing kernels of some familiar spaces.
		\begin{enumerate}
			\item For the Hardy spaces on $\mathbb{D}^d$,
			\begin{equation*}
				K_{H^2(\mathbb{D}^d)}(z,w) = \prod_{j = 1}^{d} \frac{1}{1 - z_j \overline{w}_j}.
			\end{equation*}
			
			\item For the Dirichlet space on $\mathbb{D}$,
			\begin{equation*}
				K_{\mathcal{D}}(z,w) = \frac{1}{z \overline{w}} \log \left( \frac{1}{1 - z \overline{w}} \right).
			\end{equation*}
			
			\item For the Bergman space on $\mathbb{D}$,
			\begin{equation*}
				K_{L^2_a(\mathbb{D})}(z,w) = \frac{1}{(1 - z \overline{w})^2}.
			\end{equation*}
		\end{enumerate}
	\end{example}
	
	An important property of every reproducing kernel is that it is a \emph{positive semi-definite (PSD)} function, i.e, given any $N \in \mathbb{N}$ and any $\{x_1, \dots, x_N\} \subset X$, the matrix $(K(x_j,x_k))_{N \times N}$ is PSD. Just like matrices/operators, we write $K \geq 0$ to mean $K$ is PSD. We also note that $K$ is \emph{self adjoint}, i.e.,
	\begin{equation*}
		K(x,y) = \overline{K(y,x)} \foral x, y \in X.
	\end{equation*}
	Here, the overline denotes complex conjugate. We end this subsection with the following theorem of Moore (communicated by Aronszajn in \cite[Part I Section 2 (4)]{Aro50}), which shows that these two properties determine reproducing kernels (see also \cite[Theorem 2.23]{AM02}).
	
	\begin{theorem}\label{thm:Moore-Aronsz.thm}
		If $K : X \times X \to \mathbb{C}$ is a self adjoint PSD map, then there exists a RKHS $\mathcal{H}(K) \subset \mathcal{F}(X,\mathbb{C})$ such that $K$ is the reproducing kernel of $\mathcal{H}(K)$.
	\end{theorem}
	
	\subsubsection{Multiplier algebras}\label{subsubsec:mult.alg}
	
	Every RKHS $\mathcal{H}$ comes with its \emph{multiplier algebra}
	\begin{equation*}
		\mathcal{M}(\mathcal{H}) := \{ \varphi : X \to \mathbb{C} : \varphi f \in \mathcal{H} \foral f \in \mathcal{H} \}.
	\end{equation*}
	Note that if the constant function $1 \in \mathcal{H}$, then $\mathcal{M}(\mathcal{H}) \subset \mathcal{H}$. Also, it is easy to see from the closed graph theorem that if $\varphi \in \mathcal{M}(\mathcal{H})$, then the map
    \begin{equation*}
        M_\varphi : f \mapsto \varphi f
    \end{equation*}
    is a bounded linear map from $\mathcal{H}$ to itself. If we endow $\mathcal{M}(\mathcal{H})$ with the norm
	\begin{equation*}
		\|\varphi\|_{\mathcal{M}(\mathcal{H})} := \|M_\varphi\| \foral \varphi \in \mathcal{M}(\mathcal{H}),
	\end{equation*}
	then $\mathcal{M}(\mathcal{H})$ turns into a Banach algebra.
	
	It is easy to see how the reproducing kernel $K$ of $\mathcal{H}$ relates to the multipliers. First, note that
	\begin{equation*}
		M_\varphi^* K_x = \overline{\varphi(x)} K_x \foral x \in X \AND \varphi \in \mathcal{M}(\mathcal{H}).
	\end{equation*}
	Then, it follows that
	\begin{equation*}
		\sup_{x \in X} |\varphi(x)| \leq \sup_{x \in X} \frac{\|M_\varphi^* K_x\|_\mathcal{H}}{\|K_x\|_\mathcal{H}} \leq \|\varphi\|_{\mathcal{M}(\mathcal{H})}.
	\end{equation*}
	This shows that every multiplier is a bounded function on $X$\footnote{We know from Section \ref{subsubsec:Dirich.sp} that the converse is not true for the Dirichlet space.}. Now, recall that if $T$ is a bounded operator on any Hilbert space, then
	\begin{equation*}
		\|T\| \leq 1 \iff I - T^* T \geq 0. 
	\end{equation*}
	A quick corollary of this fact is that $\varphi \in \mathcal{M}(\mathcal{H})$ with $\|\varphi\|_{\mathcal{M}(\mathcal{H})} \leq c$ if and only if
	\begin{equation}\label{eqn:mult.equiv}
		(c^2 - \varphi(x)\overline{\varphi(y)}) K(x,y) \geq 0.
	\end{equation}
	
	We wish to study the properties of functions in a RKHS through its multipliers. For instance, the multiplier algebra of $H^2(\mathbb{D})$ is $H^\infty(\mathbb{D})$, and we saw that the zero sets of every $H^2(\mathbb{D})$ is the same as that of a Blaschke product, which is an $H^\infty(\mathbb{D})$ function. This is also true for $\mathcal{D}$ and its multiplier algebra (see \cite[Theorem 2 (c)]{RS88}), but not true for $L^2_a(\mathbb{D})$, where $\mathcal{M}(L^2_a(\mathbb{D})) = H^\infty(\mathbb{D})$ (see \cite[Theorem 1]{Hor74}). The problem of determining zero sets is often related to determining \emph{interpolating sequences}. This brings us to the following question, which was first posed and answered by Pick \cite{Pic15} (and independently by Nevanlinna \cite{Nev19}) for $H^\infty(\mathbb{D})$. Given $N$ points $\{ x_1, \dots, x_N \} \subset X$, $N$ target values $\{c_1, \dots, c_N\} \subset \mathbb{C}$ and some $c > 0$, when can we find a multiplier $\varphi \in \mathcal{M}(\mathcal{H})$ with $\|\varphi\|_{\mathcal{M}(\mathcal{H})} \leq c$ such that
	\begin{equation*}
		f(x_j) = c_j \foral 1 \leq j \leq N?
	\end{equation*}
	This is called the \emph{Pick interpolation problem}. From \eqref{eqn:mult.equiv}, we get the following necessary condition for such a $\varphi$ to exist
	\begin{equation*}
		((c^2 - c_j \overline{c_k}) K(x_j,x_k))_{N \times N} \geq 0 \text{ for some } c > 0.
	\end{equation*}
	For $H^2(\mathbb{D})$, Pick \cite{Pic15} showed that this condition is also sufficient for the existence of such a $\varphi$. In this case, we say that $K$ has the \emph{Pick property}.
	
	The point of the above discussion is that one can construct $H^\infty(\mathbb{D})$ functions with a norm bound by just assigning target values on a finite set of points that help satisfy a certain PSD condition with the kernel of $H^2(\mathbb{D})$. It is now natural to ask when and for what kinds of kernels is the Pick property also sufficient.
	
	\subsubsection{Complete Pick property}\label{subsubsec:complete.Pick.prop}
	
	It is often convenient to work with a stronger version of the Pick property, called the \emph{complete Pick property}. The adjective `complete' is used when an object/property has a specific matrix-valued extension without any constraints on the size of the matrices. Fix a RKHS $\mathcal{H}$ on some set $X$ with kernel $K$. For any $n \in \mathbb{N}$, consider
	\begin{equation*}
		\mathcal{H} \otimes \mathbb{C}^n := \left\{ F = \begin{pmatrix}
			f_1 \\
			\vdots \\
			f_n
		\end{pmatrix} : f_j \in \mathcal{H} \foral 1 \leq j \leq n \right\},
	\end{equation*}
	and equip it with the inner product
	\begin{equation*}
		\left\langle \begin{pmatrix}
			f_1 \\
			\vdots \\
			f_n
		\end{pmatrix}, \begin{pmatrix}
		g_1 \\
		\vdots \\
		g_n
		\end{pmatrix} \right\rangle_{\mathcal{H} \otimes \mathbb{C}^n} := \sum_{j = 1}^n \langle f_j, g_j \rangle_{\mathcal{H}}.
	\end{equation*}
	For $m, n \in \mathbb{N}$, let $\mathcal{F}(X,\mathbb{C}) \otimes \mathbb{C}^{m \times n}$ be the collection of $m \times n$ matrix-valued functions on $X$ and define the multiplier algebra
	\begin{equation*}
		\mathcal{M}(\mathcal{H} \otimes \mathbb{C}^n, \mathcal{H} \otimes \mathbb{C}^m) :=\{ \Phi \in \mathcal{F}(X,\mathbb{C}) \otimes \mathbb{C}^{m \times n} : \Phi F \in \mathcal{H} \otimes \mathbb{C}^m \text{ if } F \in \mathcal{H} \otimes \mathbb{C}^n \}.
	\end{equation*}
	
	Following an idea similar to how we obtained \eqref{eqn:mult.equiv}, one can also show that $\Phi \in \mathcal{M}(\mathcal{H} \otimes \mathbb{C}^n, \mathcal{H} \otimes \mathbb{C}^m)$ with norm at most $C > 0$ if and only if
	\begin{equation*}
		(C^2 I_m - \Phi(x)\Phi(y)^*) K(x,y) \geq 0,
	\end{equation*}
	where $I_m$ is the $m \times m$ identity matrix. Let us then consider a matrix-valued interpolation problem, i.e., given $N$ points $\{x_1, \dots, x_N\} \subset X$, $N$ target matrices $\{C_1, \dots C_N\} \subset \mathbb{C}^{m \times n}$ and some $C > 0$, we wish to find a multiplier $\Phi \in \mathcal{M}(\mathcal{H} \otimes \mathbb{C}^n, \mathcal{H} \otimes \mathbb{C}^m)$ with norm at most $C$ such that
	\begin{equation*}
		\Phi(x_j) = C_j \foral 1 \leq j \leq N,
	\end{equation*}
	It is clear that the following is a necessary condition for such a $\Phi$ to exist
	\begin{equation*}
		((C^2 I_m - \Phi(x_j) \Phi(x_k)^*) K(x_j,x_k))_{mN \times mN} \geq 0 \text{ for some } C > 0.
	\end{equation*}
	When this condition is also sufficient for the existence of such a $\Phi$ with norm at most $C$, we say that the kernel $K$ has the complete Pick property, and we refer to it as a \emph{complete Pick kernel}.
	
	It turns out that if $K$ satisfies some minimal conditions, then it is possible to determine if it is a complete Pick kernel. We say that $K$ is \emph{irreducible} if $K(x,y) \neq 0$ for all $x,y \in X$, and $K_x$ and $K_y$ are linearly independent for any distinct $x, y \in X$
	Any irreducible kernel $K$ can be \emph{normalized} at a given point $x_0 \in X$, i.e., we can replace $K$ with another irreducible kernel $\widetilde{K}$ such that $\mathcal{H}(K)$ and $\mathcal{H}(\widetilde{K})$ are unitarily equivalent and $\widetilde{K}_{x_0} \equiv 1$ on $X$ (see \cite[Section 2.6]{AM02}). The following theorem of Agler and McCarthy \cite[Theorem 4.2]{AM00} gives the classification we need (also see \cite{McC92, McC94, Qui93}).
	
	\begin{theorem}\label{thm:class.complete.Pick.kernel}
		Let $K$ be an irreducible kernel on a set $X$ that is normalized at some $x_0 \in X$. Then, $K$ is a complete Pick kernel if and only if there exists a Hilbert space $\mathcal{K}$ and a map $\mathfrak{b} : X \to \mathcal{K}$ such that $\|\mathfrak{b}(x)\|_\mathcal{K} < 1$ for all $x \in X$, $\mathfrak{b}(x_0) = 0$ and
		\begin{equation*}
			K(x,y) = \frac{1}{1 - \langle \mathfrak{b}(y), \mathfrak{b}(x) \rangle_\mathcal{K}} \foral x,y \in X.
		\end{equation*}
		
		In particular, $K$ is a complete Pick kernel if and only if
        \begin{equation*}
            1 - \frac{1}{K} \geq 0.
        \end{equation*}
	\end{theorem}
	
	\subsubsection{Examples}\label{subsubsec:examples.complete.Pick.sp}
	
	Let us determine which of the spaces we considered so far have kernels with the complete Pick property. We refer to those RKHSs whose kernels have the complete Pick property as \emph{complete Pick spaces}.
	
	\begin{enumerate}
		\item It is clear that $H^2(\mathbb{D})$ is a complete Pick space, since
		\begin{equation*}
			1 - \frac{1}{K_{H^2(\mathbb{D})}(z,w)} = z \overline{w} \geq 0.
		\end{equation*}
		
		\item It is also clear that the $L^2_a(\mathbb{D})$ is not a complete Pick space, since
		\begin{equation*}
			1 - \frac{1}{K_{L^2_a(\mathbb{D})}(z,w)} = z \overline{w} (z \overline{w} - 2) \not\geq 0.
		\end{equation*}
		
		\item Similarly, $H^2(\mathbb{D}^d)$ is not a complete Pick space for any $d > 1$.
		
		\item In order to determine whether $\mathcal{D}$ is a complete Pick space we require the following lemma, which follows from a classical result due to Kaluza \cite{Kal28} on the coefficients of the inverse of a power-series in one variable.
		
		\begin{lemma}\label{lem:Kaluza}
			If $K$ is a PSD kernel on $\mathbb{D}$ of the form
			\begin{equation*}
				K(z,w) = \sum_{k \in \mathbb{Z}_+} c_k z^k \overline{w}^k,
			\end{equation*}
			where $c_0 = 1$, $c_k > 0$ for all $k > 0$ and
			\begin{equation*}
				\frac{c_k}{c_{k - 1}} \leq \frac{c_{k + 1}}{c_k},
			\end{equation*}
			then $1 - \frac{1}{K} \geq 0$ and hence, $K$ is a complete Pick kernel.
		\end{lemma}
		
		With the help of this lemma, it is easy to check that $\mathcal{D}_t$ is a complete Pick space for each $t > 0$. In particular, $\mathcal{D}$ is a complete Pick space.
		
		\item The most important example of a complete Pick space is the Drury--Arveson space $\mathcal{H}^2_d$ on $\mathbb{B}_d$ for $d \in \mathbb{N}$. We introduced $\mathcal{H}^2_d$ in Section \ref{subsec:Dirich.on.B_d} as the space $\mathcal{D}_{d - 1}(\mathbb{B}_d)$. It can also be defined as the RKHS with kernel\footnote{Note that both these definitions result in different inner products, but it can be checked that they generate the same topology. We omit the technical details.}
		\begin{equation*}
			K_{\mathcal{H}^2_d}(z,w) := \frac{1}{1 - \langle z,w \rangle_{\mathbb{C}^d}}.
		\end{equation*}
        
		An infinite variable version of the Drury--Arveson space is notable for being the `universal' complete Pick space. In essence, if $\mathcal{H}(K)$ is separable, then the space $\mathcal{K}$ in Theorem \ref{thm:class.complete.Pick.kernel} can instead be chosen to be $\ell^2(\mathbb{Z}_+)$, and the map $\mathfrak{b}$ can be chosen to be an embedding of $X$ into $\ell^2(\mathbb{Z}_+)$. If we now construct the Drury--Arveson space on the unit ball of $\ell^2(\mathbb{Z}_+)$, i.e. the RKHS over the unit ball of $\ell^2(\mathbb{Z}_+)$ with kernel
        \begin{equation*}
            K_{\ell^2}(v,w) := \frac{1}{1 - \langle v, w \rangle_{\ell^2(\mathbb{Z}_+)}},
        \end{equation*}
        then the map
        \begin{equation*}
            K_x \mapsto [K_{\ell^2}]_{\mathfrak{b}(x)}
        \end{equation*}
        gives an isometric linear embedding of $\mathcal{H}(K)$ into $\mathcal{H}(K_{\ell^2})$. One may think of $\mathcal{H}(K_{\ell^2})$ as an infinite variable version of the Drury--Arveson space. The reader may find more details on this topic in \cite[Chapter 8]{AM02}.
        
        It is also worth noting that infinitely many variables are needed in general, since the Dirichlet space, for instance, cannot be embedded in $H^2_d$ for any $d \in \mathbb{N}$ (see \cite[Corollary 11.9]{Har17}). Moreover, Hartz \cite{Har22} later showed that this cannot be achieved even under a weaker hypothesis.
		
		\item An example of a complete Pick space that is not a space of analytic functions is the \emph{Sobolev space} on the unit interval $[0,1] \subset \mathbb{R}$, defined as
		\begin{equation*}
			\mathcal{W}^{1,2} := \left\{ f : [0,1] \to \mathbb{C} : \int_0^1 (|f(t)|^2 + |f'(t)|^2) dt < \infty \right\}.
		\end{equation*}
		The reproducing kernel of $\mathcal{W}^{1,2}$ is given by
		\begin{equation*}
			K_{\mathcal{W}^{1,2}}(t,s) := \frac{2e}{e^2 - 1} \cosh(1 - \max{\{t,s\}}) \cosh(\min{\{t,s\}}).
		\end{equation*}
		It is not at all easy to verify using Theorem \ref{thm:class.complete.Pick.kernel} that $\mathcal{W}^{1,2}$ is a complete Pick space. See \cite[Theorem 7.43]{AM02} for a detailed proof of this fact.
	\end{enumerate}
	
	\subsection{Multiplier-cyclicity}\label{subsec:mult.cyc}
	
	Let $\mathcal{H}$ be a RKHS on some set $X$ with kernel $K$. We say that $f \in \mathcal{H}$ is \emph{multiplier-cyclic} if
	\begin{equation*}
		\mathcal{M}[f] := \overline{f \mathcal{M}(\mathcal{H})} = \mathcal{H}.
	\end{equation*}
	If we let
    \begin{equation*}
        \mathcal{M}_\mathcal{H} := \{ M_\varphi : \varphi \in \mathcal{M}(\mathcal{H}) \} \subset \mathcal{L}(\mathcal{H}),
    \end{equation*}
    then note that $\mathcal{M}_{\mathcal{H}}$ is $\mathcal{M}(\mathcal{H})$-compatible (as in Section \ref{subsec:cyclicity}) and that $\mathcal{M}_\mathcal{H}$-cyclicity is the same as multiplier-cyclicity. Let us now mention a generalization of Proposition \ref{prop:poly.cyc.iff.irred.fac.cyc} for multiplier-cyclicity.

	\begin{proposition}\label{prop:mult-cyc.of.prod}
		Let $\mathcal{M}(\mathcal{H})$ be dense in $\mathcal{H}$, and let $\varphi \in \mathcal{M}(\mathcal{H})$ and $f \in \mathcal{H}$. Then, $\varphi f$ is multiplier-cyclic if and only if both $\varphi$ and $f$ are multiplier-cyclic.
	\end{proposition}
	
	\begin{proof}
		Suppose $\varphi, f$ are multiplier-cyclic. Let $g \in \mathcal{H}$ and $\epsilon > 0$ be arbitrary. Since $\varphi$ is multiplier-cyclic, there exists $\psi_0 \in \mathcal{M}(\mathcal{H})$ such that
		\begin{equation}\label{equation_first step}
			\| \psi_0 \varphi - g \| < \frac{\epsilon}{2}.
		\end{equation}
		Now, since $f$ is also multiplier-cyclic, there exists $\psi \in \mathcal{M}(\mathcal{H})$ such that
		\begin{equation*}
			\| \psi f - \psi_0 \| < \frac{\epsilon}{2 \| \varphi \|_{\mathcal{M}(\mathcal{H})}}.
		\end{equation*}
		It follows that
		\begin{equation}\label{equation_second_step}
			\| \psi \varphi f - \psi_0 \varphi \| < \frac{\epsilon}{2}.
		\end{equation}
		Combining (\ref{equation_first step}) and (\ref{equation_second_step}), we get
        \begin{equation*}
            \| \psi \varphi f - g \| < \epsilon.
        \end{equation*}
        As $\epsilon > 0$ and $g \in \mathcal{H}$ were arbitrarily chosen, we get that $\varphi f$ is multiplier-cyclic.
		
		Conversely, suppose $\varphi f$ is multiplier-cyclic, and note that
		\begin{equation*}
			\overline{f \mathcal{M}(\mathcal{H})} \supset \overline{f \left (\varphi \mathcal{M}(\mathcal{H}) \right )} = \overline{\varphi f \mathcal{M}(\mathcal{H})} = \mathcal{H}.
		\end{equation*}
		Thus, $f$ is multiplier-cyclic. To show that $\varphi$ is multiplier-cyclic as well, note that
		\begin{equation*}
			\overline{\varphi \mathcal{M}(\mathcal{H})} \supset \varphi \mathcal{H} \supset \varphi f \mathcal{M}(\mathcal{H}).
		\end{equation*}
		Since $\varphi f$ is multiplier-cyclic, it follows that $\varphi$ is multiplier-cyclic as well.
	\end{proof}
	
	We now list some important properties of the multiplier algebra of a complete Pick space that follow from Theorem \ref{thm:class.complete.Pick.kernel} (see \cite[Proposition 2.1]{CHMR22}).
	
	\begin{proposition}\label{prop:prop.mult.alg.cPsp}
		Let $\mathcal{H}$ be a complete Pick space with a normalized irreducible kernel $K$, and let $K$, $\mathcal{K}$ and $\mathfrak{b}$ be as in the statement of Theorem \ref{thm:class.complete.Pick.kernel}. Then,
		\begin{enumerate}
			\item The function
            \begin{equation*}
                \mathfrak{b}_x : y \mapsto \langle \mathfrak{b}(y), \mathfrak{b}(x) \rangle_\mathcal{K}
            \end{equation*}
            is a multiplier with $\|\mathfrak{b}_x\|_{\mathcal{M}(\mathcal{H})} < 1$ for all $x \in X$.
			
			\item For all $x \in X$, we have
            \begin{equation*}
                K_x = \frac{1}{1 - \mathfrak{b}_x}
            \end{equation*}
            Thus, $K_x \in \mathcal{M}(\mathcal{H})$ and $K_x$ is multiplier-cyclic for each $x \in X$.
			
			\item $\mathcal{M}(\mathcal{H})$ is dense in $\mathcal{H}$.
		\end{enumerate}
	\end{proposition}

	\begin{remark}\label{rem:non.cPness.of.RKHS}
		For a general RKHS $\mathcal{H}$, not only can we not guarantee that $\mathcal{M}(\mathcal{H})$ is dense, but it may not contain any non-constant functions at all (see \cite[Section 2]{AM02}). Moreover, it is not true for a general RKHS that the kernel functions are multiplier-cyclic (see \cite[Theorem 13]{Izu17}). Thus, Proposition \ref{prop:prop.mult.alg.cPsp} provides ways to check whether a given RKHS with a normalized irreducible kernel is not a complete Pick space.
	\end{remark}
	
	\subsubsection{Smirnov representation}\label{subsubsec:Smirnov.repn}
	
	We mentioned a result of F. and R. Nevanlinna \cite{NN22} in Section \ref{subsec:Hardy.on.D} (see the references under Theorem \ref{thm:radial.limits.H^p(D)}) which states that every $f \in H^2(\mathbb{D})$ can be written as $f = \varphi/\psi$ for some $\varphi, \psi \in H^\infty(\mathbb{D})$. Combining this fact with the Smirnov factorization theorem, it follows that we can choose $\psi$ such that it is outer. This is what is called a \emph{Smirnov representation of $f$}. Taking motivation from this we define the \emph{Smirnov class} of a RKHS $\mathcal{H}$ as
	\begin{equation*}
		\mathcal{N}^+(\mathcal{H}) := \left\{ f = \frac{\varphi}{\psi} : \varphi, \psi \in \mathcal{M}(\mathcal{H}) \AND \psi \text{ is multiplier-cyclic} \right\}.
	\end{equation*}
	It is then of interest to ask when we can guarantee the following inclusion
    \begin{equation*}
        \mathcal{H} \subset \mathcal{N}^+(\mathcal{H}).
    \end{equation*}
    
    First, we note that this is not true in general. Consider $L^2_a(\mathbb{D})$ and recall that its multiplier algebra is $H^\infty(\mathbb{D})$. This means that every function in $\mathcal{N}^+(L^2_a(\mathbb{D}))$ has radial limits a.e. on $\mathbb{T}$ (by Fatou's theorem \cite{Fat06}). However, it is possible to construct an $L^2_a(\mathbb{D})$ function that does not have radial limits at any point on $\mathbb{T}$ (e.g. use \cite[Theorem 1]{Mac62}). We therefore conclude that
    \begin{equation*}
        L^2_a(\mathbb{D}) \not\subset \mathcal{N}^+(L^2_a(\mathbb{D})).
    \end{equation*}
    The following result of Aleman, Hartz, McCarthy and Richter \cite[Theorem 1.1]{AHMR17} settles the case of complete Pick spaces.
	
	\begin{theorem}\label{thm:Smirnov.rep.cPsp}
		Let $\mathcal{H}$ be a complete Pick space on $X$ with an irreducible kernel $K$ that is normalized at some $x_0 \in X$, and let $f \in \mathcal{H}$ be such that $\|f\| \leq 1$. Then, there exist $\varphi, \psi \in \mathcal{M}(\mathcal{H})$ such that
        \begin{enumerate}
            \item $\|\varphi\|_{\mathcal{M}(\mathcal{H})} \leq 1$,

            \item $\psi(x_0) = 0$, and
            
            \item $\|\psi\|_{\mathcal{M}(\mathcal{H})} \leq 1$,
        \end{enumerate}
        for which
        \begin{equation*}
            f = \frac{\varphi}{1 - \psi}.
        \end{equation*}
		
		Moreover, $1 - \psi$ is multiplier-cyclic for any $\psi \in \mathcal{\mathcal{H}}$ that satisfies (2) and (3) above. Thus, $\mathcal{H} \subset \mathcal{N}^+(\mathcal{H})$ is always true for a complete Pick space.
	\end{theorem}
	
	Jury and Martin \cite[Corollary 3.4]{MJ21} used completely different techniques to obtain the same result but with a different flavor of Smirnov representation.
	
	\begin{theorem}\label{thm:Jury--Martin.Smirn.rep}
		Let $\mathcal{H}$ be a complete Pick space with a normalized irreducible kernel. Then, for any $f \in \mathcal{H}$ there exist $\varphi, \psi \in \mathcal{M}(\mathcal{H})$ such that $1/\psi \in \mathcal{H}$ and
        \begin{equation*}
            f = \frac{\varphi}{\psi}.
        \end{equation*}
	\end{theorem}
	
	It is obvious that if $\psi \in \mathcal{M}(\mathcal{H})$ is such that $1/\psi \in \mathcal{H}$ then $\psi$ is multiplier-cyclic. Thus, Jury and Martin's result also provides a Smirnov representation for functions in complete Pick spaces.
	
	Let us turn to some easy applications of the above theorems. First, we noticed in Section \ref{subsubsec:mult.alg} (see the paragraph after \eqref{eqn:mult.equiv}) that the zero sets of functions in $H^2(\mathbb{D})$ is the same as that of its multiplier algebra, but this is not the case for $L^2_a(\mathbb{D})$. One immediately sees from Theorems \ref{thm:Smirnov.rep.cPsp} and \ref{thm:Jury--Martin.Smirn.rep} that the zero sets of functions in a complete Pick space are the same as that of its multiplier algebra. We also obtain the following corollary using Proposition \ref{prop:mult-cyc.of.prod} and Theorem \ref{thm:Jury--Martin.Smirn.rep}.
	
	\begin{corollary}\label{cor:mult.in.cPsp}
		Let $\mathcal{H}$ be a complete Pick space with a normalized irreducible kernel, and let $f = \varphi/\psi$ be as in Theorem \ref{thm:Jury--Martin.Smirn.rep}. Then, $f$ is multiplier-cyclic if and only if $\varphi$ is multiplier-cyclic.
	\end{corollary}
	
	The above corollary shows that in order to identify all multiplier-cyclic functions in a complete Pick space, we only need to identify multipliers that are multiplier-cyclic, and also multipliers that are invertible. With this in mind, let us now establish certain algebraic properties of multiplier-cyclic functions.
	
	\subsubsection{Algebraic properties of multiplier-cyclic functions}\label{subsubsec:alg.prop.mult-cyc.func}
	
	A large part of our previous discussion was based around the properties \ref{item:NI}, \ref{item:Mult} and \ref{item:Inv} for the Hardy spaces. We gave examples of when these properties generalize and, if not, then whether there are known counter-examples in this direction. However, these properties -- as mentioned -- only made sense for families of spaces. Let us begin by defining a localized version of these properties and summarize what we know from before. Let $\mathcal{H}$ be a RKHS such that its kernel is normalized at some point (so, in particular, $1 \in \mathcal{H}$). Then, we say that $\mathcal{H}$ satisfies
	\begin{enumerate}
		\myitem[\textbf{Mult}]\label{item:LocMult} if for every $f_1, f_2 \in \mathcal{H}$ such that $f_1 f_2 \in \mathcal{H}$, we have $f_1 f_2$ is multiplier-cyclic if and only if $f_1, f_2$ are both multiplier-cyclic.
	
		\myitem[\textbf{Inv}]\label{item:LocInv} if for every $f \in \mathcal{H}$ such that $1/f \in \mathcal{H}$, we know that $f$ is multiplier-cyclic.
	\end{enumerate}
	Clearly, \ref{item:LocMult} implies \ref{item:LocInv} by taking $f_1 = f$ and $f_2 = 1/f$ and thus, a counter-example to \ref{item:LocInv} serves as a counter-example to \ref{item:LocMult}. Also, Proposition \ref{prop:mult-cyc.of.prod} shows that if we restrict to the case when either of $f, f_1, f_2$ is a multiplier, then \ref{item:LocMult} and \ref{item:LocInv} will hold for any RKHS $\mathcal{H}$.
	
	We saw that both these properties hold for $H^2(\mathbb{D})$, but we do not know the answer for $H^2(\mathbb{D}^2)$ (see Problems \ref{prob:ni.Hardy.d>1}--\ref{prob:inv.Hardy.d>1}). A famous result of Borichev and Hedenmalm \cite[Theorem 1.4]{BH97} shows that neither of these properties hold for $L^2_a(\mathbb{D})$. In fact, they construct an explicit example of a function $f \in L^2_a(\mathbb{D})$ for which $1/f \in L^2_a(\mathbb{D})$, but it is not shift-cyclic. They also highlight (see \cite[p. 763]{BH97}) that \ref{item:MC} holds for $L^2_a(\mathbb{D})$, i.e., multiplier-cyclicity is equivalent to shift-cyclicity of the above $f$. The following result by the author settles the case of complete Pick spaces (see \cite[Theorem 2.5.4]{Sam22}).
	
	\begin{theorem}\label{thm:mult.inv.for.cPsp}
		Let $\mathcal{H}$ be a complete Pick space with a normalized irreducible kernel. Then, both \ref{item:LocMult} and \ref{item:LocInv} hold for $\mathcal{H}$.
	\end{theorem}
	
	\begin{proof}
		By the previous discussion, it suffices to show that \ref{item:LocMult} holds. Thus, let $f_1, f_2 \in \mathcal{H}$ be such that $f_1, f_2 \in \mathcal{H}$ as well, and let
        \begin{equation*}
            f_1 = \frac{\varphi_1}{\psi_1} \AND f_2 = \frac{\varphi_2}{\psi_2}
        \end{equation*}
        be as in Theorem \ref{thm:Jury--Martin.Smirn.rep}. Note that
		\begin{align*}
			f_1 f_2 \mathcal{M}(\mathcal{H}) &\supset \frac{\varphi_1 \varphi_2}{\psi_1 \psi_2} \mathcal{M}(\mathcal{H}) \\
            &\supset \frac{\varphi_1 \varphi_2}{\psi_1 \psi_2} (\psi_1 \psi_2 \mathcal{M}(\mathcal{H})) \\
            &\supset \varphi_1 \varphi_2 \mathcal{M}(\mathcal{H}).
		\end{align*}
		Now, if $f_1$ and $f_2$ are both multiplier-cyclic then $\varphi_1 \varphi_2$ is multiplier-cyclic as well, since both $\varphi_1$ and $\varphi_2$ are multiplier-cyclic (follows from Proposition \ref{prop:mult-cyc.of.prod} and Corollary \ref{cor:mult.in.cPsp}). Consequently, Corollary \ref{cor:mult.in.cPsp} shows that $f_1 f_2$ is multiplier-cyclic.
		
		Conversely, suppose $f_1 f_2$ is multiplier-cyclic. By Corollary \ref{cor:mult.in.cPsp}, it suffices to show that $\varphi_1$ and $\varphi_2$ are multiplier-cyclic. Now, note that
		\begin{equation*}
			(\varphi_1 \varphi_2) \left( \frac{1}{\psi_1} \right) = \psi_2 (f_1 f_2).
		\end{equation*}
		Since $\psi_2$ is obviously multiplier-cyclic, we can use Proposition \ref{prop:mult-cyc.of.prod} to show that the R.H.S. above is multiplier-cyclic. Now, since $1/\psi_1 \in \mathcal{H}$, we can use Proposition \ref{prop:mult-cyc.of.prod} to conclude that $\varphi_1 \varphi_2$ is multiplier-cyclic. Finally, using Proposition \ref{prop:mult-cyc.of.prod} once again, we get that both $\varphi_1$ and $\varphi_2$ are multiplier-cyclic.
	\end{proof}
	
	Of course, there are several other functions spaces that one might wish to consider. Therefore, we close this section with the following open problems.
	
	\begin{problem}\label{prob:mult.inv.in.general}
		Classify all kernels $K$ such that $\mathcal{H}(K)$ satisfies \ref{item:LocMult} and \ref{item:LocInv}.
	\end{problem}
	
	\begin{problem}\label{prob:mult-cyc.equiv.shift-cyc}
		Classify all Hilbert spaces of analytic functions (say, over $\mathbb{D}^d$ or $\mathbb{B}_d$) for which multiplier-cyclicity is equivalent to shift-cyclicity.
	\end{problem}
		
	
	\section{Cyclicity preserving operators}\label{sec:CPO}
	
	Up to this point, we were exclusively interested in determining when functions/polynomials in a space are shift/multiplier-cyclic. In cases when we cannot answer with certainty, we were interested in understanding what sort of properties do these functions/polynomials enjoy. In the last part of our discussion, we take a slightly unorthodox approach and try to understand other objects that are derived from shift-cyclicity. To be precise, suppose $\mathcal{X}$ and $\mathcal{Y}$ are Banach spaces of analytic functions on some domains (possibly in different dimensions). In this section, we wish to study bounded linear operators $T : \mathcal{X} \to \mathcal{Y}$ that preserve shift-cyclicity, i.e. $Tf$ is shift-cyclic in $\mathcal{Y}$ whenever $f$ is shift-cyclic in $\mathcal{X}$. We refer to such a $T$ as a \emph{cyclicity preserving operator}. Let us motivate this problem through an example that we already saw.
	
	Recall Theorem \ref{thm:ex.outer.but.not.shift-cyc} above, which is a result of Rudin \cite[Theorem 4.4.8]{Rud69} that shows that the function
	\begin{equation*}
		f_o(z_1,z_2) := \exp \left( \frac{z_1 + z_2 + 2}{z_1 + z_2 - 2} \right)
	\end{equation*}
	is outer but not shift-cyclic in $H^2(\mathbb{D}^2)$. The proof of this fact is indirect and interesting in its own right. Roughly, the steps are as follows:
	\begin{enumerate}
		\item Show that the operator $T_{out} : H^\infty(\mathbb{D}) \to H^2(\mathbb{D}^2)$ given by
		\begin{equation*}
			T_{out}f(z_1, z_2) = f\left( \frac{z_1 + z_2}{2} \right)
		\end{equation*}
		preserves outer functions.
		
		\item Show that the operator $T_{cyc} : H^2(\mathbb{D}^2) \to H^{\frac{1}{4}}(\mathbb{D})$ given by
		\begin{equation*}
			T_{cyc}f(z) = f\left( \frac{1 + z}{2}, \frac{1 + z}{2} \right)
		\end{equation*}
		preserves shift-cyclic functions.
		
		\item Show that $f_o$ is the image under $T_{out}$ of an outer function, but the image under $T_{cyc}$ of $f_o$ is not shift-cyclic.
	\end{enumerate}
	
	Gibson, Lamoureux and Margrave \cite{GLM11} had a very different motivation for studying cyclicity preserving operators $T : H^2(\mathbb{D}) \to H^2(\mathbb{D})$, coming from a geophysical modelling and signal processing point of view. They noted that all such operators must be \emph{weighted composition operators}, i.e., there exist holomorphic maps $\varphi : \mathbb{D} \to \mathbb{D}$ and $\psi : \mathbb{D} \to \mathbb{C}$ such that
	\begin{equation*}
		Tf = \psi \cdot (f \circ \varphi) \foral f \in H^2(\mathbb{D}).
	\end{equation*}
	Mashreghi and Ransford gave a different proof of this result using the \emph{Gleason--Kahane--{\.Z}elazko theorem} \cite{Gle67, KZ68}, and generalized it to Banach spaces of analytic functions on $\mathbb{D}$ (see \cite[Theorem 3.2]{MR15}).
	
	\begin{theorem}\label{thm:cyc.preserv.op.on.D}
		Let $\mathcal{X}, \mathcal{Y}$ be Banach spaces of analytic functions on $\mathbb{D}$. If $T : \mathcal{X} \to \mathcal{Y}$ is a bounded cyclicity preserving operator, then
		\begin{equation*}
			Tf = \psi \cdot (f \circ \varphi) \foral f \in \mathcal{X},
		\end{equation*}
		where $\psi = T1$ is shift-cyclic and $\varphi = Tz/T1 \in \operatorname{Hol}(\mathbb{D})$.
	\end{theorem}
	
	In order to extend this theorem at the level of generality we seek, we need take a detour and introduce some more notation.
	
	\subsection{Maximal domains}\label{subsec:max.dom}
	
	Fix $d \in \mathbb{N}$ and let $\mathcal{X}$ be a Banach space of analytic functions on some open set $\Omega \subset \mathbb{C}^d$. Recall that $\mathcal{P}_d$ is the set of all polynomials in $d$ complex variables. We define the \emph{maximal domain of $\mathcal{X}$} as
	\begin{equation*}
		\Omega_{max} := \left\{ w \in \mathbb{C}^d : \Lambda_w (P) := P(w), \forall P \in \mathcal{P}_d \text{ extends continuously to } \mathcal{X} \right\}.
	\end{equation*}
 
    We borrow this terminology from a similar concept related to \emph{algebraic consistency}. The reader can find more details on this topic in \cite[Section 5]{Har17} and \cite[Section 2]{MS17}.
	In essence, $\Omega_{max}$ is the largest domain over which the point-evaluations $\Lambda_w$ make sense for all functions in $\mathcal{X}$. The following result of the author \cite[Theorem 3.2]{Sam21} establishes an important geometric property of $\Omega_{max}$. Recall that the \emph{right Harte spectrum $\sigma_r(T)$} of a $d$-tuple of commuting operators $T = (T_1,\dots, T_d) \in \mathcal{L}(\mathcal{X})$ is a compact set (see \cite[Theorem 14.53 (v)]{AM02}), given by the complement in $\mathbb{C}^d$ of
	\begin{equation*}
		\rho_r(T) := \left\{ w \in \mathbb{C}^d : \exists \{A_j\} \subset \mathcal{L}(\mathcal{X}) \text{ such that } \sum_{j = 1}^d (T_j - w_j I) A_j = I \right\}.
	\end{equation*}
	
	\begin{theorem}\label{thm:max.dom.in.rHspec}
		If $\mathcal{X}$ is a Banach space of analytic functions on $\Omega \subset \mathbb{C}^d$, then\footnote{A recent result of Mironov and the author \cite{MS24} shows that $\Omega_{max} = \sigma_p(S^*)$, i.e, the collection of \emph{joint eigenvalues} of the $d$-tuple $S^* = (S_1^*,\dots, S_d^*)$ of the topological adjoint of $S_j$'s.}
		\begin{equation*}
			\Omega \subseteq \Omega_{max} \subseteq \sigma_r(S).
		\end{equation*}
		
		In particular, $\Omega_{max}$ is a bounded set in $\mathbb{C}^d$.
	\end{theorem}
	
	One usually finds that $\sigma_r(S) = \overline{\Omega}$. Then, it is not difficult to determine the maximal domain as one only needs to identify $w \in \partial \Omega$ for which there is an $f \in \mathcal{X}$ with a singularity at $w$. In this case, we conclude that $w \not\in \Omega_{max}$. Using this idea, it is straightforward to check that the maximal domain of $H^p(\mathbb{D}^d)$ is $\mathbb{D}^d$ for all $1 \leq p < \infty$ and $d \in \mathbb{N}$. Similarly, for $\mathcal{D}_t(\mathbb{D}^d)$ one can check that
	\begin{equation*}
		\Omega_{max} = \begin{cases}
			\mathbb{D}^d &\foral t \leq 1,\\
			\overline{\mathbb{D}^d} &\foral t > 1. 
		\end{cases}
	\end{equation*}
	See the recent preprint of Mironov and the author \cite{MS24} for more details and other interesting examples. We also provide the following alternate description of $\Omega_{max}$ and relate it to shift-cyclicity of polynomials (see \cite[Theorem 3.1]{MS24})\footnote{It is interesting to note that this result holds even if $\mathcal{X}$ is just a topological vector space. We do not need local convexity of $\mathcal{X}$ for this either.}.
	
	\begin{theorem}\label{thm:max.dom.descr}
		Let $\mathcal{X}$ be a Banach space of analytic functions on $\Omega \subset \mathbb{C}^d$. Then, $w \in \Omega_{max}$ if and only if
        \begin{equation*}
            S[z_1 - w_1, \dots, z_d - w_d] \neq \mathcal{X},
        \end{equation*}
        where
		\begin{equation*}
			S[z_1 - w_1, \dots, z_d - w_d] = \overline{\operatorname{span}}_{1 \leq j \leq d} S[z_j - w_j].
		\end{equation*}
	\end{theorem}
	
	For $d = 1$, we combine the above description of $\Omega_{max}$ with the final assertion of Proposition \ref{prop:poly.cyc.iff.irred.fac.cyc} to solve \ref{item:SCP} in the one variable case.
	
	\begin{corollary}\label{cor:shift-cyc.poly.1-var}
		Let $\mathcal{X}$ be a Banach space of analytic functions on $\Omega \subset \mathbb{C}$. Then, a polynomial $P$ is shift-cyclic if and only if
        \begin{equation*}
            P(w) \neq 0 \foral w \in \Omega_{max}.
        \end{equation*}
	\end{corollary}
	
	Thus, we clearly need to keep the maximal domain in mind when talking about shift-cyclicity in a general Banach space of analytic functions $\mathcal{X}$. For this reason, we note that every such $\mathcal{X}$ can be thought of as consisting of functions that are defined over each $w \in \Omega_{max}$, since we may just define
	\begin{equation*}
		f(w) := \Lambda_w(f) \foral w \in \Omega_{max} \setminus \Omega.
	\end{equation*}
	
	\subsection{Classification of cyclicity preserving operators}\label{subsec:main.res.cpo}
	
	Let $\mathcal{X}$ be a Banach space of analytic functions on $\Omega \subset \mathbb{C}^d$. It is straightforward to check that
    \begin{equation*}
        e^{w \cdot z} := \exp\left( \sum_{j=1}^d w_j z_j \right) \in \mathcal{X} \foral w \in \mathbb{C}^d.
    \end{equation*}
    Indeed, this follows by checking that the truncations of the Taylor series of $e^{w \cdot z}$ converge in the norm of $\mathcal{X}$ to $e^{w \cdot z}$. One can also show that $e^{w \cdot z}$ is shift-cyclic in $\mathcal{X}$ by using Proposition \ref{prop:equiv.def.of.shift.cyc} and noting that
    \begin{equation*}
        \lim_{n \to \infty} \|1 - P_n e^{z \cdot w}\|_\mathcal{X} = 0,
    \end{equation*}
    where $P_n$'s are the truncations of the Taylor series of $e^{-w \cdot z}$. The following theorem by the author \cite[Theorem 2.1]{Sam21} is the first step in identifying cyclicity preserving operators.
	
	\begin{theorem}\label{thm:preliminary.res}
		Suppose $\mathcal{X}$ is a Banach space of analytic functions on $\Omega \subset \mathbb{C}^d$. If $\Lambda$ is a bounded linear functional on $\mathcal{X}$ such that
        \begin{equation*}
            \Lambda(e^{w \cdot z}) \neq 0 \foral w \in \mathbb{C}^d,
        \end{equation*}
        then there exists $c \in \mathbb{C} \setminus \{0\}$ and some $w_0 \in \Omega_{max}$ such that $\Lambda = c \Lambda_{w_0}$.
	\end{theorem}
	
	Now, if $T : \mathcal{X} \to \mathcal{Y}$ is a bounded cyclicity preserving operator between two Banach spaces of analytic functions, then it is clear that $\Lambda := \Lambda_w \circ T$ satisfies the hypothesis of the above theorem for each $w$ in the maximal domain of $\mathcal{Y}$. This immediately gives us the first part of the following characterization of cyclicity preserving operators in general (see \cite[Theorem 1.5]{Sam21}).
	
	\begin{theorem}\label{thm:CPO}
		Suppose $\mathcal{X}, \mathcal{Y}$ are Banach spaces of analytic functions on $\Omega \subset \mathbb{C}^d$ and $\Omega' \subset \mathbb{C}^{d'}$ respectively. Then, every bounded cyclicity preserving operator $T : \mathcal{X} \to \mathcal{Y}$ is of the form
		\begin{equation}\label{eqn:weight.comp.op}
			Tf = \psi \cdot (f \circ \varphi) \foral f \in \mathcal{X},
		\end{equation}
		where $\psi = T1 \in \mathcal{Y}$ is shift-cyclic and $\varphi = (\frac{Tz_1}{T1},\dots,\frac{Tz_d}{T1}) : \Omega'_{max} \to \Omega_{max}$.
		
		Moreover, if $\mathcal{Y} = H^p(\mathbb{D}^{d'})$ for some $1 \leq p < \infty$, then every bounded operator as in \eqref{eqn:weight.comp.op} preserves cyclicity whenever $\psi$ is shift-cyclic.
	\end{theorem}
	
	The second part of the above theorem holds for any $\mathcal{Y}$ that satisfies:
    \begin{enumerate}
    \item the multiplier algebra of $\mathcal{Y}$ is $H^\infty(\mathbb{D}^{d'})$, and
    
    \item multiplier-cyclicity is equivalent to shift-cyclicity in $\mathcal{Y}$.

    \end{enumerate}
    However, we do not have a satisfactory answer in general. Therefore, we arrive at the following problem.
 
	\begin{problem}
		For what spaces $\mathcal{Y}$ does the second part of Theorem \ref{thm:CPO} hold?
	\end{problem}
	
	It is also important to note that the first part of Theorem \ref{thm:CPO} holds for \emph{outerness preserving operators} between Hardy spaces as well\footnote{Of course, we need to replace `$\psi$ is shift-cyclic' with `$\psi$ is outer' here.}. However, one can check that the operator $T_{cyc}$ from Rudin's example at the beginning of our discussion in Section \ref{sec:CPO} is a (weighted) composition operator that does not preserve outer functions. Indeed, the function $f_o$ from the example is an outer function whose image is not outer. Therefore, we arrive at the following question.
	
	\begin{problem}
		Classify all bounded weighted composition operators between Hardy spaces that preserve outer functions.
	\end{problem}
	
	Lastly, we note that a result of Chu, Hartz, Mashreghi and Ransford \cite[Theorem 1.2]{CHMR22} generalizes Theorem \ref{thm:preliminary.res} -- in some sense -- to complete Pick spaces (where maximal domains need to be considered as in \cite{Har17, MS17}). Thus, we close our discussion with one final open problem.
	
	\begin{problem}
		For any two complete Pick spaces $\mathcal{H}$ and $\mathcal{K}$, describe all operators $T : \mathcal{H} \to \mathcal{K}$ that preserve multiplier-cyclicity.
	\end{problem}
	
	
	\section*{Acknowledgements} The author would like to thank Orr Moshe Shalit for motivating him to write this article and for providing some important remarks and corrections in the previous draft of the article. Thanks also to Agniva Roy for many helpful suggestions on how to improve the introduction, and to Daniel Seco for his input on the content surrounding Dirichlet-type spaces.
    
    \bibliographystyle{plain}
    \bibliography{SCAFS}
\end{document}